\documentclass{aims}

\usepackage{pgfplots}
\usepackage{amsmath,amssymb,mathrsfs}
\usepackage{blkarray}
 \usepackage{amscd}
  \usepackage{paralist}
    \usepackage{bbm}
  \usepackage{epsfig,graphics} 
  \usepackage{epsfig} 
\usepackage{graphicx}  
\usepackage{epstopdf}
 \usepackage[colorlinks=true]{hyperref}
 \usepackage{multirow}

   \usepackage{tikz}
\usetikzlibrary{arrows}
\usetikzlibrary{positioning}
\hypersetup{urlcolor=blue, citecolor=red}

\def\currentvolume{X}
 \def\currentissue{X}
  \def\currentyear{200X}
   \def\currentmonth{XX}
    \def\ppages{X--XX}
     \def\DOI{10.3934/xx.xx.xx.xx}

\newtheorem{theorem}{Theorem}[section]
\newtheorem{corollary}{Corollary}

\newtheorem{proposition}{Proposition}

\theoremstyle{definition}
\newtheorem{definition}[theorem]{Definition}
\newtheorem{remark}{Remark}

\def\rd{{\rm d}}

\def\ve {{\bf e}}

\def\vp{{\bf p}}
\def\vpi{{\boldsymble{\pi}}}
\def\vmu{{\boldsymble{\mu}}}

\def\vx{{\bf x}}
\def\vu{{\bf u}}
\def\vv{{\bf v}}

\def\mF{{\bf F}}

\def\vpi{\mbox{\boldmath$\pi$}}
\def\vPi{\mbox{\boldmath$\Pi$}}

\def\vmu{\mbox{\boldmath$\mu$}}
\def\Pr{\mathop{\rm Pr}\nolimits}
\def\dbar{\mathchar'26\mkern-12mu \rd}

\def\( {\left( }
\def\) {\right) }

\newcommand{\beq}{\begin{equation}}
\newcommand{\eeq}{\end{equation}}
\newcommand{\ba}{\begin{array}}
\newcommand{\ea}{\end{array}}
\newcommand{\bea}{\begin{eqnarray*}}
\newcommand{\eea}{\end{eqnarray*}}
\newcommand{\bc}{\begin{center}}
\newcommand{\ec}{\end{center}}

\newcommand{\bt}{\begin{table}}
\newcommand{\et}{\end{table}}

\title[Stochastic Dynamics II]{Stochastic Dynamics II: Finite Random Dynamical Systems, Linear Representation, and Entropy Production}

\author[F. X.-F. Ye and H. Qian]{}


 \email{xye16@jhu.edu}
  \email{hqian@u.washington.edu}
  
\begin{document}

\maketitle

\centerline{\scshape Felix X.-F. Ye }
\medskip
{\footnotesize
 \centerline{Department of Applied Mathematics \& Statistics}
   \centerline{Johns Hopkins University}
   \centerline{Baltimore, MD 21218-2608, USA}
} 

\medskip

\centerline{\scshape Hong Qian}
\medskip
{\footnotesize
 \centerline{Department of Applied Mathematics}
   \centerline{University of Washington}
   \centerline{Seattle, WA 98195-3925, USA}
}

\bigskip


\begin{abstract}
We study finite state random dynamical systems (RDS)
and their induced Markov chains (MC)  as stochastic models 
for complex dynamics.  The linear representation of 
deterministic maps in RDS is a matrix-valued random variable 
whose expectation corresponds to the transition matrix of 
the MC.  The instantaneous Gibbs entropy, Shannon-Khinchin 
entropy of a step, and the entropy production rate of the MC are discussed.  These three concepts, as key anchoring points in 
applications of stochastic dynamics, characterize respectively the uncertainties of a system at instant time $t$, the randomness 
generated in a step in the dynamics, and the dynamical asymmetry 
with respect to time reversal.  The stationary entropy production 
rate, expressed in terms of the cycle distributions, has found 
an expression in terms of the probability of
the deterministic maps with single attractor in the maximum 
entropy RDS.  For finite RDS with invertible transformations, 
the non-negative entropy production rate of its
MC is bounded above by the Kullback-Leibler divergence
of the probability of the deterministic maps with respect to its time-reversal
dual probability.
\end{abstract}

\section{Introduction}
\label{sec-1}
	The theory of nonlinear stochastic dynamical systems  
has gradually replacing classical deterministic dynamics as 
the mathematical representation of complex  systems 
and processes \cite{mqy-chinese,mumford,qian-I}.  Depending 
on the origin of uncertainties in the applications, 
 stochastic dynamics can be mathematically modeled
either in terms of  stochastic process or random 
dynamical system (RDS) \cite{Arnold1998}.  In \cite{ywq}, we have 
studied the contradistinctions between these two types of
mathematical approaches, and found the RDS perspective
as a more refined description of stochastic phenomena.
In the present paper we continue to study discrete time
finite state RDS, with a certain level of rigor,
 and several related
concepts that are likely to be the key anchor points
between stochastic dynamics and their applications.
These concepts were originated in statistical physics as the 
theory that justifies thermodynamics via a mechanical
formulation; they have since permeated through discussions 
on complexity.

The stochastic process perspective, particularly in terms of
Markov processes, has a long history in 
statistical physics \cite{vankampen,risken}.   For example,
equations describing continuous time Markov process, e.g., master equations and Fokker-Planck equations, have been studied extensively. On the other hand,
the field of deterministic nonlinear dynamics has
witnessed a surge of activities in terms of 
Perron-Frobenius-Ruelle operator (or transfer operator) \cite{lasota2013,baladi} 
and Koopman operator \cite{mezic} as the linear
representations of nonlinear dynamics.  In terms of
finite-state RDS and its induced Markov chain (MC), the transition 
probability matrix of the latter has been identified as the
expectation of a matrix-valued random variable, where the 
matrix is the linear representation of the deterministic map in RDS \cite{ywq}: 
The left- and right-matrix-multiplications represent Perron-Frobenius
and Koopman operators of the deterministic map, respectively;  
and families of stochastic Perron-Frobenius operators and stochastic Koopman operators can be defined for MC in terms of the 
language of RDS. 
Both families form semi-groups whose generators are represented by the transition matrix in terms of left- and right-multiplications.  They are discrete analog of the solutions to Kolmogorov forward and backward equations.

	Entropy and entropy production are two distinctly different
key concepts originated in thermodynamics, the study of
Newtonian particles in terms of their stochastic 
motions | called heat.  In physicists' theory, entropy is a 
function of the state of a system. Entropy production, 
however, is associated with the amount of heat being
generated in a process; it is path dependent in general.  
In fact, the physicists of the earlier time carefully introduced
the notations of $\rd A$ and $\ \dbar Q$, where 
$\rd A$ represents a change in a state function $A$ 
that is path independent, and $\dbar Q$ is associated 
with the accumulation of heat $Q$ (or work $\dbar W$) 
that is a function of a path.   The celebrated First Law of 
Thermodynamics states that $\ \dbar Q +\ \dbar W= \rd E$, 
where $E$ is
called internal energy.  We see that if
the work $\ \dbar W = \mF\cdot\rd\vx$ is due to a 
force $\mF$ with a potential, $\mF=-\nabla U(\vx)$,
then $\dbar Q = \rd\big(E+U\big)$. 
In terms of the nonlinear stochastic dynamics,
therefore, entropy, as a state function, should be 
a functional of the instantaneous probability distribution
$\vp(t)$, but entropy production is associated
with the transition probability. 

	Entropy is also a widely used concept in statistical physics, information theory, and many other areas that involve statistics 
and distributions.  Yet, it still does not have a universal 
measure-theoretical definition.  In the present work, we 
introduce the concept of {\em information} as the negative logarithm of
the non-negative random variable, the Radon-Nikodym derivative $\frac{\rd \mathbb{P}}{\ \rd \mathbb{P}'}(\omega)$
defined on a given probability 
space $(\Omega,\mathcal{F},\mathbb{P})$, 
 where the probability measure $\mathbb{P}'$ is 
absolutely continuous w.r.t. $\mathbb{P}$.
 So the information is a dimensionless, non-negative random variable, which, 
we argue, is the only legitimate quantity that can be 
placed inside the logarithmic function: 
$-\log\left[\frac{\rd \mathbb{P}}{\  \rd \mathbb{P}'}(\omega)\right]$.
This definition is motivated by the work of Kolmogorov \cite{kolmogorov}, the idea of self-information \cite{tribus},
and recent success of stochastic thermodynamics
\cite{seifert}.  The {\it relative entropy}, or {\it Kullback-Leibler divergence}, then simply is the negative expected value of 
the information:
\[
      H\big(\mathbb{P},\mathbb{P}'\big) = \mathbb{E}^{\mathbb{P}}\left[\log
                   \left(\frac{\rd \mathbb{P}}{\  \rd \mathbb{P}'}(\omega)\right)\right]  
           = -\left\{ -\int_{\mathbb{P}'} 
               \left[\frac{\rd \mathbb{P}}{\  \rd \mathbb{P}'}(\omega)\right]
             \log\left[\frac{\rd \mathbb{P}}{\  \rd \mathbb{P}'}(\omega)\right]
            \rd \mathbb{P}'(\omega) \right\}.
\]  
The term inside $\{\cdots\}$ actually could be understood 
as Shannon's information entropy of $\mathbb{P}$, expressed as 
the Radon-Nikodym derivative w.r.t. $\mathbb{P}'$.  For
a finite ranged continuous random variable $X$: $(\Omega,\mathcal{F},\mathbb{P})\rightarrow ([a,b], \mathcal{B})$, in terms of
the normalized Lebesgue measure on the finite interval, $\mu=\frac{\lambda([a,b])}{b-a}$,
the relative entropy $H(\mu, \mathbb{P})$ quantifies {\em statistical bias} in 
using the observable $X$ to represent the probability space 
$(\Omega,\mathcal{F},  \mathbb{P})$.

	Under this definition, standard Shannon's information 
entropy has a hidden reference measure: 
the counting measure for a discrete 
random variable and the Lebesgue measure for a continuous 
random variable.  When the normalization is finite, the 
Shannon entropy is off by a trivial constant; but the constant 
could be problematic since normalization of a reference measure, based on observables from engineering and science, 
may have a dimension.  On the other hand, if the reference 
measure is not normalizable, then many difficulties 
were known to arise \cite{hobson,gaspard-wang}.
We shall point out that in statistical physics, the notion
of free energy is simply the relative entropy w.r.t. the Gibbs 
measure, and the normalization factors have a prominent
role known as {\em partition functions}.

	This measure-theoretic notion of information and 
 entropy is naturally generalized from
random variables to stochastic processes.  The information (or randomness) generated in a step, then, gives
rise to the concept of {\em stochastic entropy production rate},
which has several forms based on the reference
measure.  More specifically, in stochastic dynamics, the 
amount of information (or randomness) in the entire history 
of a path is different from the the amount of information
(or randomness) in the system at the ``current time'' $t$.
The term ``entropy'' physicists use often refers to the latter; 
they call the former entropy production:  Entropy as a
state function is a functional of the marginal
probability at time $t$ irrespective of path-history, and 
as in the theory of thermodynamics, entropy production
is path-dependent.  When the historical paths of 
two ``identical'' particles are neglected, they become 
indistinguishable \cite{bennaim}.

	We have inherited, therefore, two very 
different classes of stochastic, thermodynamic quantities: 
entropy and relative entropy associated with $\vp(t)$ in 
one class, and entropy production, Shannon-Khinchin 
entropy, etc. defined in terms of the path probability 
$\mathbb{P}_{[0,t]}$, in another.  These latter
path-dependent quantities, say $\Theta[\mathbb{P}_{[0,t]}]$,
naturally defines ``production rate'' and ``in a step'' as $\ \dbar\Theta\equiv\Theta[\mathbb{P}_{[0,t+1]}]-\Theta[\mathbb{P}_{[0,t]}]$.
But they should not be confused with the change in $\Xi[\vp(t)]$: $\Delta\Xi\equiv\Xi[\vp(t+1)]-\Xi[\vp(t)]$,
where $\Xi[\vp(t)]$ belongs to the first class.
Only in certain special types of systems, for example 
systems with detailed balance, 
that the entropy production rate can be expressed in terms of 
the change of a state function.  It is also immediately clear that 
the stationary entropy production rate is zero in this type of 
systems, that mathematically represent physicists' notion
of thermodynamic equilibrium, in which one can find 
{\em thermodynamic potential function} 
for path-dependent quantities.

	The entropy and entropy production introduced above
can be rigorously established in the theory of Markov
chains.  The notions of  instantaneous Gibbs entropy, 
Shannon-Khinchin entropy in a step \cite{khinchin}, and 
entropy production rate \cite{jqq} are three distinct
concepts, each represents a different aspect of the same
stochastic dynamics. 	In the present work, we are interested in these concepts under the representation of finite i.i.d. RDS. In particular, we 
establish an inequality between entropy production rate of 
a doubly stochastic MC and the relative entropy of an RDS that consists of 
all invertible transformations.

	The paper is organized as follows: In Sec. \ref{sec-2} we
provide the definition of a general RDS with a medium level of 
rigor, and provide some simple examples in finite state space.  We particularly
call the attention of the difference in the habitual
perspectives of mathematics, in terms of space of all
paths, and that of statistical physics, in terms of
evolving probability distribution along the time,
on the state space.   In Sec. \ref{sec-3}, we discuss 
the linear representations of deterministic maps in an 
RDS and its corresponding MC.  Sec. \ref{sec-4} first
provides a brief, but rather coherent presentation of the 
theory of entropy production of MC; and then
establishes several interesting relationships about entropy production between
the MC and its corresponding maximum 
entropy RDS, and the doubly stochastic MC and its
 invertible RDS.

\section{Preliminaries}
\label{sec-2}
In \cite{ywq}, we have presented a finite-state i.i.d. RDS intuitively. 
It is described by the triplet $(\mathscr{S}, \Gamma, Q)$, where $\mathscr{S}$ is a finite state space, $\Gamma$ is the set of all deterministic transformations from $\mathscr{S}$ into itself and $Q$ is the probability measure on $\sigma$-field of $\Gamma$. Note $\Gamma$ is a monoid with the composition of transformations 
as the operation. If the finite state space $\mathscr{S}$ has $n$ state, then there are $n^n$ possible deterministic transformations. Therefore the cardinality $\|\Gamma\|= n^n$.  As a dynamics
in the state space, the system starts initially with some state $i$ in $\mathscr{S}$, maps $\alpha_0, \alpha_1, \dots$ in $\Gamma$ are independently chosen according to the probability measure $Q$. The random variable $X_t$ is constructed by means of composition of independent random maps, $X_t= \alpha_{t-1}\circ \cdots \circ \alpha_0(i)$. 

	Mathematically, one follows the construction to rigorously define 
the RDS \cite{arnold2013, Arnold1991,rds-finance}.  We shall start with general RDS and later specify each 
term for finite RDS, and focus mainly on finite  i.i.d. case. Here finite RDS means RDS on the finite state space. 

\begin{definition} 
$(\Omega, \mathcal{F}, \mathbb{P}, \theta)$ is a {\it metric\footnote{The term metric is often used in the literature for historical reasons.} dynamical system} if $(\Omega, \mathcal{F}, \mathbb{P})$ is a probability space and $\theta(t): \Omega\rightarrow\Omega, t\in \mathbb{Z}$ is a family of measure-preserving transformations such that 
\begin{enumerate}
\item $\theta(0)=id, \theta(s)\circ \theta(t)=\theta(s+t)$ for every $s,t\in \mathbb{Z}$.

\item The mapping $(t, \omega)\rightarrow \theta(t)\omega$ is measurable.

\item $\theta(t)\mathbb{P}=\mathbb{P}$ for every $t\in \mathbb{Z}$.
\end{enumerate}

\end{definition}
The set of the map $\theta(t)$ forms a commutative group and preserves the measure $\mathbb{P}$.  Distinctly different from physicists' notion of dynamics as ``step-by-step'' motion, in
stochastic mathematics, the space $\Omega$ contains all the possible paths, and $(\Omega, \mathcal{F}, \mathbb{P}, \theta(t))$ is a 
stationary process. This two-sided discrete-time dynamical system $(\Omega, \mathcal{F}, \mathbb{P}, \theta(t))$ is also known as {\it base flow} of random dynamical system. In many applications, the base flow is usually ergodic. If the property 3 is not fulfilled, then $(\Omega, \mathcal{F}, \mathbb{P}, \theta(t))$ is called {\it measurable dynamical system}.   Non-stationary dynamics belong to the latter, as 
illustrated next.

\begin{definition}\label{RDS-def}
A measurable {\it random dynamical system} (RDS) on the complete separate metric space $(\mathscr{S}, d)$ over a metric dynamical system $(\Omega, \mathcal{F}, \mathbb{P}, \theta)$ is a map with one-sided time, $  \mathbb{N}\times \Omega\times \mathscr{S}\rightarrow \mathscr{S}: (t, \omega, i) \rightarrow \varphi(t, \omega)i$,
with the following properties: 
\begin{enumerate}
\item The map $ (t, \omega, i) \rightarrow \varphi(t, \omega)i$ is $\mathcal{B}(\mathbb{N})\otimes \mathcal{F}\otimes \mathcal{B}(\mathscr{S}), \mathcal{B}(\mathscr{S})$-measurable. 

\item The map $i\rightarrow\varphi(t,\omega)i$ satisfies the {\it cocycle property}: 
\begin{align} \label{cocycle}
 \varphi(0,\omega)=id, \varphi(s+t, \omega)=\varphi(s, \theta(t)\omega)\circ \varphi(t, \omega) 
 \end{align}
for every $s, t\in \mathbb{N}$ and $\omega\in \Omega$. 

\end{enumerate}
\end{definition}

From the definition, the RDS is driven by the base flow and for one particular noise realization $\omega$, one can treat $i\rightarrow\varphi(t, \omega)i$ as a non-autonomous dynamical system, which defines one-point motion. The cocycle property is intuitively understood as follows: evolve some initial state $i$ for $s$ steps with particular noise realization $\omega$ and then go through $t$ more steps with the same noise from the $s$ steps mark; it gives the same result as evolving the same initial state $i$ for $t+s$ steps with the same noise realization $\omega$. 
The map $\varphi(t, \omega)$ may not be invertible, so the RDS is defined one-sided in time. 
We call an RDS {\it ergodic} if there exists a probability measure $\vpi$ on $\mathscr{S}$, such that for any $i\in \mathscr{S}$, the law of the one-point motion $\varphi(t, \omega)i$ converges to $\vpi$. We don't assume this one-point motion is Markovian. 
It is possible to relax the metric dynamical system to measurable dynamical system, but the limiting behaviors of RDS will be unclear. 

\subsection{ Examples of RDS}

\vskip 0.3cm\noindent

{\bf\em Example 1.} 
For finite i.i.d. RDS, the above terms have explicit expressions. The space $\Omega$ is the full shift, 
$\Omega=\Gamma^{\mathbb{Z}}$, which is the set of all possible 
two sided infinitely long sequences of deterministic transformations. 
\begin{align} \label{omega}
\Omega=\Big\{  \omega: (\cdots \alpha_{-1},\alpha_0, \alpha_1, \alpha_2 \cdots, \alpha_k, \cdots) \big| \alpha_k\in \Gamma \Big\}
\end{align}
The probability measure is the Bernoulli measure defined on the cylinder set, 
\begin{align}
\mathbb{P}\big([\alpha_0, \alpha_1, \alpha_2, \dots, \alpha_k]\big)=Q(\alpha_0)Q(\alpha_1)\dots Q(\alpha_k).
\end{align}
So maps at different steps are chosen independently with the same probability measure $Q$. The mappings $\theta(t)$ are the left Bernoulli shift for $t$ elements, i.e, $\theta(t)\omega=\alpha_{t}$. Define the time-one mapping $\varphi(1, \omega)=\alpha_0$ which is the first element of the sequence of deterministic transformations. Then the map $\varphi(t, \omega)$ is the composition of i.i.d. random maps, $ \alpha_{t-1}\circ \cdots \circ \alpha_0$. If it applies to an initial state $i$, it generates a one-point motion $X_t(\omega)=\varphi(t,\omega)i$. Now we constructs finite i.i.d. RDS rigorously. Clearly $X_t$ is a Markov chain (MC) and its transition probability is
\begin{align} \label{transition}
\Pr(i, G)=Q(\alpha: \alpha(i)\in G),
\end{align}
 for any $x\in \mathscr{S}$ and any measurable set $G\in \mathcal{B}(\mathscr{S})$. 
 
 The connection between finite i.i.d. RDS and MC is discussed in \cite{ywq}. Briefly, a finite i.i.d. RDS uniquely defines an MC, but a given MC is generally compatible with many possible RDS. The reason for non-uniqueness is that a transition probability only determines the statistical property of the one-point motion of a possible RDS, while an RDS also describes the simultaneous motion of two or more points. 
Given a general MC (not necessarily with finite states), Kifer proved the existence of corresponding i.i.d. RDS representation by measurable maps with some weak conditions on the state space \cite{kifer2012random}. Quas also showed the sufficient conditions for the representation of an MC on a manifold by smooth maps \cite{quas1991}. 

  {\bf\em Example 2.} 
Another example to generate random maps is via a Markov chain. Then the probability measure $\mathbb{P}$ is the Markov measure. The measure of a cylinder set is defined by 
\begin{align}
\mathbb{P}([\alpha_0, \alpha_1, \dots, \alpha_k])=\vpi_{\alpha_0}p_{\alpha_0\alpha_1}\cdots p_{\alpha_{k-1}\alpha_k},
\end{align}
where $p_{\alpha_i \alpha_j}$ is the transition probability of the Markov chain from the map $\alpha_i$ to $\alpha_j$ and $\vpi$ is the stationary probability of the Markov chain. One can check this Markov measure is still invariant with the Bernoulli shift map. 
However, the stochastic process induced $X_t=\varphi(t, \omega)i$ may not be a Markov chain in general. Even one keeps as many steps of memory as possible, the Markov property may not hold any more \cite{ywq}. We use `may' since in Example 4 we will show it is still possible that the dynamics of $X_t$ to be Markovian.

  {\bf\em Example 3.} 
It is also possible to generate random maps via an independent but not identical process. 
Then the measure is defined by 
\begin{align}
\mathbb{P}([\alpha_0, \alpha_1, \alpha_2, \dots, \alpha_k])=Q_0(\alpha_0)Q_1(\alpha_1)\dots Q_k(\alpha_k),
\end{align}
where $Q_0, Q_1, \dots $ might be different measures. Shift maps $\theta(t)$ in general doesn't preserve this measure and is no longer stationary. So it is only the measurable dynamical system. However, the stochastic process $X_t=\varphi(t, \omega)i$ is still well-defined and follows a time-inhomogeneous Markov chain with its transition probability at step $t$
\begin{align}
\Pr_t(i, G)=Q_t(\alpha: \alpha(i)\in G).
\end{align}
Unless some special cases, different probability measures $Q_t$ will result in different transition probability $P_t$. 

From these three examples, it seems that independence of random maps at each step may be necessary to the Markov property of the stochastic process $X_t$. It turns out that's not true. Here is the counter-example. In fact, if we choose random maps in Markovian way, the state dynamics could still be Markovian. 

  {\bf\em Example 4.} 
Let the state space be $\mathscr{S}=\{1,2\}$ and the set of deterministic transformations $\Gamma$ be 
\begin{align*}
  \Gamma=\left\{
    \underbrace{\left(\begin{array}{c}
        1\rightarrow 1   \\   2 \rightarrow 2
          \end{array}\right)}_{\alpha_1},\  \underbrace{\left(\begin{array}{cc}
          1\rightarrow 2 \\    2\rightarrow 1
          \end{array}\right)}_{\alpha_2}, \  \underbrace{\left(\begin{array}{cc}
       1\rightarrow 1\\   2\rightarrow 1
          \end{array}\right)}_{\alpha_3},\  \underbrace{\left(\begin{array}{cc}
 1\rightarrow 2\\    2\rightarrow 2
          \end{array}\right)}_{\alpha_4}
\right\}.
\end{align*}

\begin{figure}
\includegraphics[scale=0.5]{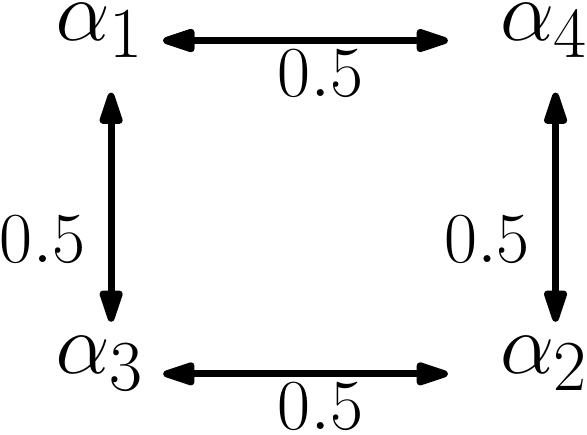}
\caption{Random maps are generated via this Markov chain. This is the illustration of state transition diagram for the Markov chain. }
\label{figure-markov}
\end{figure}

They are denoted as $\alpha_1, \alpha_2, \alpha_3, \alpha_4$. Then an MC with the transition matrix 
\[
M=\left(\begin{array}{cc}
       0.5 & 0.5   \\  0.5 & 0.5
          \end{array}\right)
\]
can be represented by i.i.d. RDS with probability $Q(\alpha_1)=0.2, Q(\alpha_2)=0.2, Q(\alpha_3)=0.3, Q(\alpha_4)=0.3$. It is possible to generate the random maps via a Markov chain, but the dynamics on the state space are still Markovian. If the initial distribution of the deterministic maps is $\vp_{\alpha_1}(0)=\vp_{\alpha_2}(0)=0.5$ and the state transition diagram is illustrated in Fig. \ref{figure-markov}, then $\vp_{\alpha_1}(t)=\vp_{\alpha_2}(t), \vp_{\alpha_3}(t)=\vp_{\alpha_4}(t)$ at any steps. So the RDS induces a Markov chain in the state space with the transition matrix $M$. On the other hand, if we consider two-point motion $X_0=1$, $Y_0=2$ and apply the same sequence of maps, it is impossible to have $X_0=1, X_1=1, X_2=2$ and $Y_0=2, Y_1=2, Y_2=1$ since $\alpha_1$ cannot go to $\alpha_2$ in the single step.

\section{Linear Representation of Finite RDS}
\label{sec-3}

First, an RDS on an $n$-dimensional vector space $X$ is called {\it a linear RDS} if $\varphi(t,\omega)\in \mathcal{L}(X)$ for each $t\in \mathbb{N}, \omega\in\Omega$, where $\mathcal{L}(X)$ is the space of linear operators of $X$. If state $i$ is denoted as standard basis $\ve_i$ in $n$-dimensional vector space $\mathrm{R}^n$, the deterministic transformation $\alpha\in\Gamma$ has a linear representation in the $n\times n$ matrices, called {\it deterministic transition matrix}
\begin{align} \label{matrix-rep}
(P)_{ij}\triangleq\begin{cases}1,&j=\alpha(i),
\\0,&\text{otherwise},
\end{cases}
\ \ i,j \in \mathscr{S}
\end{align}
The dynamics of the map $\alpha$ applying on the state $i$ is represented by the multiplication $\ve_i P_\alpha$. Note that $\ve_i$ is a row vector. Moreover, $\ve_i$ can be considered as the probability concentrated on state $i$.
Such representation is also discussed in \cite{ywq}. Now $P_\alpha$ is a 0-1 matrix and has exactly one entry 1 in each row and 0s otherwise. 
 So $P_\alpha$ is the representation of $\varphi(t,\omega)$ in the space of linear operators of $\mathbb{R}^n$. 

Second, composition of transformations is represented by the matrix multiplication, i.e, $P_{\alpha_1}\cdot P_{\alpha_2}=P_{\alpha_1\circ \alpha_2}$. In addition, this linear RDS $\varphi(t, \omega)$ has the form of random matrices production and it is easy to see the cocycle property (\ref{cocycle}).

The stochastic process $X_t$ starting from $X_0=i$ is $X_t(\omega)=\varphi(t, \omega)i$, and its linear representation is 
\begin{align}
\vv(t)=\ve_{i} P_{\alpha_0}\cdot P_{\alpha_1}\cdots P_{\alpha_{t-1}}.
\end{align}
 It is defined in the push-forward sense. Define another stochastic process $Y_t$ starting from $Y_0=i$,  $Y_t(\omega)=\varphi(t, \theta(-t)\omega)x_0$ and its linear representation is 
 \begin{align}
\vu(t)=\ve_{i} P_{\alpha_{-t}}\cdot P_{\alpha_{-2}}\cdots P_{\alpha_{-1}}.
\end{align}
$Y_t(\omega)$ is defined in the pullback sense.
If the RDS is i.i.d. and ergodic, $X_t(\omega)$ follows an MC and $Y_t(\omega)$ has the same distribution as $X_t(\omega)$ for each $t$. But $X_t(\omega)$ and $Y_t(\omega)$ have different behaviors: $X_t(\omega)$ moves ergodically through the state space $\mathscr{S}$ along $t$; $Y_t(\omega)$ could converge to a limit as $t\rightarrow +\infty$. Similar idea was discussed in iterative random functions \cite{Diaconis1999}. 
Although $P_{\alpha_i}$ is picked randomly for $Y_t$, it is multiplied on the left hand side which is the beginning of the matrices sequence, and the rest matrices remain the same. Roughly speaking, the random matrix multiplication may have the memory decay effect along the time and the last couple matrices which are fixed may determine the vector $\vu(t)$.   Here is an elementary example to illustrate the significant
difference between a push-forward matrix multiplication and a
pullback matrix multiplication:  Consider $3\times 3$ deterministic
transition matrice and their random products:  If the matrix 
\[
            P^* = \left(\begin{array}{ccc}
             0 & 1 & 0 \\   0 & 1 & 0 \\  0 & 1 & 0
           \end{array}\right)
\]
is chosen, then the product of any deterministic 
transition matrix multiplied on the left of $P^*$ will be invariant.
This is not the case if a deterministic transition matrix is 
multiplied on the right of $P^*$:
\[
	P^* \left(\begin{array}{ccc}
             0 & 0 & 1 \\   1 & 0 & 0 \\  0 & 1 & 0
           \end{array}\right) = \left(\begin{array}{ccc}
             1 & 0 & 0 \\   1 & 0 & 0 \\  1 & 0 & 0
           \end{array}\right),  \
     P^* \left(\begin{array}{ccc}
             1 & 0 & 0 \\   0 & 0 & 1 \\  0 & 0 & 1
           \end{array}\right) = \left(\begin{array}{ccc}
             0 & 0 & 1 \\   0 & 0 & 1 \\  0 & 0 & 1
           \end{array}\right).
\]
The pullback product has a limit, while the push forward
product does not.  More rigorous discussions in terms of 
multiplicative ergodic theorem and possible extension to 
countable states can be found in \cite{hqwyy}. 

Consider $Y_t(\omega)$ starting with the whole state space $\mathscr{S}$, $Y_t(\omega)=\varphi(t, \theta(-t)\omega)\mathscr{S}$, which are $n$ simultaneous sequences starting with state $1, \dots, n$, $\text{supp}(Y_t(\omega))$ is non-increasing. Once multiple sequences collide at some instance, they will be together forever. For a fixed $\omega$, $\lim_{t\rightarrow +\infty}\text{supp}(Y_t(\omega))$ may be smaller than $\mathscr{S}$, even can be a singleton (with some assumptions).  This means these $n$ simultaneous sequences synchronize into one sequence. 
If the support of $Y_t(\omega)$  as $t\rightarrow+\infty$ is almost surely a singleton, it is equivalent with the RDS synchronizes, i.e, for any different initial states $x_1$ and $x_2$, i.e, $x_1\ne x_2$, $\lim_{t\rightarrow +\infty}\Pr(\omega: \varphi(t, \omega)x_1=\varphi(t,\omega)x_2)=1$. Then the limit $Y_{\infty}(\omega)=\lim_{t\rightarrow+\infty} Y_t(\omega)$ exists almost surely and is $\omega$-dependent. Moreover, $Y_{\infty}(\omega)$ follows the invariant distribution $\vpi$ since $X_t(\omega)$ follows the invariant distribution $\vpi$ as $t\rightarrow +\infty$,  but the limit doesn't exist. This is exactly the idea of {\it the coupling from the past} \cite{Wilson1996}. The algorithm works is because $Y_{\infty}(\omega)$ can be sampled in finite time. There exists some finite $t_0(\omega)$ such that $\varphi(t, \theta(-t)\omega)\mathscr{S}$ is singleton for all $t\ge t_0(\omega)$ almost surely. Then this singleton is $Y_{\infty}(\omega)$ for this given $\omega$ and exactly has the law of $\vpi$. So this method is also called perfect sampling. However, not every finite RDS have such properties and the sufficient condition is the RDS is monotone and ergodic \cite{Michael2008}. Except for sampling, synchronization in RDS has also been widely discovered in applied science \cite{Lin2013, ymq}.


At last,  from the definition of Perron-Frobenius-Ruelle operator (or transfer operator) for deterministic map $\alpha$, $F: \mathbb{R}^n\rightarrow \mathbb{R}^n$, 
$(F\vv)_j=\sum_{i: \alpha(i)=j}\vv_i, \vv\in \mathbb{R}^n$.
So $P_\alpha$ is the representation of Perron-Frobenius operator  for the deterministic transformation $\alpha$, 
and $\vv\rightarrow\vv P_\alpha$ can also be interpreted as the evolution of probability mass $\vv$ corresponding to the mapping $\alpha$. 
From the definition of Koopman operator for $\alpha$, $K: \mathbb{R}^n\rightarrow \mathbb{R}^n$, $(K\vu)_j=\vu_{\alpha(j)}, \vu\in \mathbb{R}^n$. So $P_\alpha^T$ is the representation of Koopman operator for $\alpha$. 
We introduce the stochastic Perron-Frobenius operator family and stochastic Koopman operator family associate to the finite RDS. 

\begin{definition}
The {\it stochastic Perron-Frobenius operator} $F_{s, t}: \mathbb{R}^n\rightarrow \mathbb{R}^n$ for every $0\le s\le t$ associate to finite RDS $\varphi$ is defined by
\begin{align} \label{P-F-operator}
(F_{s,t}\vv)_j\triangleq\mathbb{E}^\mathbb{P}\Big[\sum_{ i: \varphi(t-s, \theta(s)\omega)i =j}\vv_i\Big].
\end{align}

The {\it stochastic Koopman operator} $K_{s,t}: \mathbb{R}^n\rightarrow \mathbb{R}^n$ for every $0\le s\le t$ associate finite RDS $\varphi$ is defined by
\begin{align}\label{Koopman-operator}
(K_{s,t}\vu)_j\triangleq\mathbb{E}^\mathbb{P}\Big[\vu_{\varphi(t-s, \theta(s)\omega)j}   \Big].
\end{align}
\end{definition}
 The expectation is taken with respect to probability measure $\mathbb{P}$. We refer the family of operators $F_{s, t}, K_{s,t}$, parametrized by time $s$ and $t$, as the stochastic Perron-Frobenius operator family and the stochastic Koopman operator family respectively. 
x
If $\theta$ is the stationary process, $K_{s,t}$ and $F_{s,t}$ are also stationary, i.e, both family of operators can be expressed by the time difference, $K_{t-s}$ and $F_{t-s}$. Furthermore, if the stochastic process $X_t$ is an MC, then both family of operators form semigroups, i.e, one-parameter family of linear operators with the properties, $F_0=id$ and $F_{t+s}=F_{ t}\circ F_{s}$. Here the semigroups are characterized through their generators, $F_1$ and $K_1$, which are corresponding stochastic Perron-Frobenius operator and Koopman operator for time-one random map $\varphi(1, \omega)$.  In terms of matrix representation (\ref{matrix-rep}), the generator $M$ is represented by $M=\mathbb{E}^Q[P_\alpha]$, which is exactly the same as Markov transition matrix for $X_t$. The operator composition is also represented by matrix multiplication, so the stochastic Perron-Frobenius operator $F_t=M^t$. The stochastic Koopman operator is then represented by the adjoint of the matrix $M$. More importantly, this adjoint property is also true for general case. 

\begin{theorem}\label{adjoint}
For every $\vv, \vu\in \mathbb{R}^n$, 
\begin{align}
\langle F_{s,t}\vv, \vu\rangle =\langle \vv, K_{s,t}\vu \rangle, 
\end{align}
where $\langle\vv,\vu \rangle=\vv\vu^*$.

\end{theorem}
\begin{proof}
We first check $\vv=\ve_i$ and $\vu=\ve_j$. 
\begin{align}
&\langle F_{s,t}\ve_i, \ve_j\rangle= \mathbb{E}^\mathbb{P}\Big[\sum_{ i: \varphi(t-s, \theta(s)\omega)i =j}\ve_i\Big]=\mathbb{P}\Big[\omega:\varphi(t-s, \theta(s)\omega)i =j \Big] \\
&\langle \ve_i, K_{s,t}\ve_j \rangle=\mathbb{E}^\mathbb{P}\Big[(\ve_j)_{\varphi(t-s, \theta(s)\omega)i}   \Big]=\mathbb{P}\Big[\omega:\varphi(t-s, \theta(s)\omega)i =j \Big] 
\end{align}
Both operators are linear, so the adjoint property is true for any vector $\vv, \vu\in \mathbb{R}^n$. 
\end{proof}

\section{Entropy Theory of MC}
\label{sec-4}

If the $n$-state MC $X_t$ with the transition probability matrix 
$M_{ij}=\Pr\{X_{t+1}=j|X_t=i\}$ is irreducible and aperiodic, there exists a unique stationary distribution $\vpi$ and for any initial distribution $\vp(0)$, the MC will converge to the stationary distribution, i.e, $\lim_{t\rightarrow +\infty}\vp(t)=\lim_{t\rightarrow +\infty} \vp(0)M^t=\vpi$. 

The Shannon entropy for the probability measure $\vp$ is the expectation of the information content, 
\begin{align}\label{shannon}
S(\vp)\triangleq\mathbb{E}^\vp[-\log(\vp(\omega))].
\end{align}

 As we have
stated in Sec. \ref{sec-1}, more rigorously one considers 
the relative entropy or Kullback-Leibler divergence 
of $\vp$ with respect to $\vmu$, $H(\vp, \vmu)$,
\begin{align}
H(\vp, \vmu)\triangleq \begin{cases}\mathbb{E}^\vp \left[\log\left(\frac{\rd\vp}{\rd\vmu}(\omega)\right)\right] & \vp \ll \vmu, \\
          +\infty &\text{Otherwise,} 
\label{info-entropy}
\end{cases}
\end{align}

\subsection{Relative Entropy w.r.t. Stationary Probability}  The results in this subsection are collected from scattered 
literatures.  We give a brief summary for completeness.   
Statistical physicists always consider the instantaneous
 distribution $\vp(t)$. 
One natural choice of the $\vmu$ in (\ref{info-entropy}) 
is the invariant distribution $\vpi$.  Then we have 
\cite{voigt,qian-jmp}:

\begin{theorem}  \label{theo-voigt}
$H\big(\vp(t),\vpi\big)$ is a non-increasing function of $t$.
\end{theorem}
\begin{proof}
For $t\ge 1$, 
\begin{eqnarray*}
	\Delta H\big(\vp(t-1),\vpi\big)&\equiv& H\big(\vp(t),\vpi\big) - H\big(\vp(t-1),\vpi\big)
\\
   &=& \sum_{i\in  \mathscr{S}} \vp_i(t)\ln\left(\frac{\vp_i(t)}{\vpi_i}\right)
  -\sum_{i\in \mathscr{S}}\vp_i(t-1)\ln\left(\frac{\vp_i(t-1)}{\vpi_i}\right)
\\
   &=& \sum_{i,j\in  \mathscr{S}} \left[ \vp_j(t-1)M_{ji}\ln\left(\frac{\vp_i(t)}{\vpi_i}\right)
  - \vp_i(t-1)M_{ij}\ln\left(\frac{\vp_i(t-1)}{\vpi_i}\right) \right]
\\
   &=& \sum_{i,j\in  \mathscr{S}}\vp_i(t-1)M_{ij}\ln\left(\frac{\vpi_i \vp_j(t)}{\vp_i(t-1)\vpi_j}\right)
\\
	&\le&  \sum_{i,j\in  \mathscr{S}} \vp_i(t-1)M_{ij}\left(\frac{\vpi_i \vp_j(t)}{\vp_i(t-1)\vpi_j}-1\right)
\\
	&=&  \sum_{i,j\in  \mathscr{S}}\frac{\vpi_i M_{ij}\vp_j(t)}{\vpi_j}- \sum_{i,j\in \mathscr{S}} \vp_i(t-1)M_{ij} \ = \  1-1  \ = \ 0.
\end{eqnarray*}
\end{proof}

\begin{remark}
For a finite MC with uniform stationary 
$\vpi_i=1/n$, $S(\vp(t)) = \ln(n) - H\big(\vp(t),\vpi\big)$.
Therefore, the above theorem becomes the statement
``entropy never decreases''.  This scenario is known as 
{\em microcanonical system} in statistical physics.
\end{remark}

\begin{remark}
For any MC, $\Delta S(\vp(t)) \equiv S(\vp(t+1))-S(\vp(t))$ satisfies
\begin{eqnarray}
    \Delta S(\vp(t)) &=& \sum_{i,j\in  \mathscr{S}} \vp_i(t)M_{ij}\ln\left(
          \frac{\vp_i(t) }{\vp_j(t+1) }\right)
\label{ebe}\\
	&=&  \underbrace{ \sum_{i,j\in  \mathscr{S}}\vp_i(t)M_{ij}\ln\left(
          \frac{\vp_i(t)M_{ij} }{\vp_j(t+1)M_{ji}}\right)  }_{\text{non-negative}}+
           \sum_{i,j\in  \mathscr{S}}\vp_i(t)M_{ij}\ln\left(
          \frac{M_{ji} }{M_{ij}}\right).
\nonumber
\end{eqnarray}
If the MC is detailed balance, $\vpi_iM_{ij}=\vpi_jM_{ji}$, then
the second term on the right-hand-side can be expressed as 
$\Delta \overline{E}(\vp(t))\equiv\overline{E}(\vp(t+1))-\overline{E}(\vp(t))$, 
which is defined as
\begin{equation}
  \overline{E}(\vp(t)) = \sum_{i \in\mathscr{S}} 
          \vp_i(t) \Big( -\ln\vpi_i\Big).
\label{b-law}
\end{equation}
Note that $\Delta S$ and $\Delta \overline{E}$ are changes in functions 
of state, $S(t)$ and $\overline{E}(t)$. This scenario is known 
as Gibbsian {\em canonical system} in statistical physics.  
$\overline{E}$ should be identified with the internal energy;
and the relationship (\ref{b-law}) between internal energy 
and equilibrium measure is known as the {\em Boltzmann 
distribution}.  
\end{remark}

\begin{remark}
In statistical physics, $S$ is called Gibbs entropy.  Then
$\overline{E}-S$ should be identified with the notion of
{\em free energy} there.  Theorem \ref{theo-voigt} thus becomes
``free energy of a canonical system never increases; it
reaches its minimum when a system is at its equilibrium''.
\end{remark}

	The identification of the non-negative term in
(\ref{ebe}) with the concept of {\em entropy production 
rate} in nonequilibrium thermodynamics appeared repeatedly 
in physics and chemistry literature, see
\cite{cox,lebowitz,schnakenberg,luo-jl}.
For MC without detailed balance, the last
term in (\ref{ebe}) cannot be expressed as the change
of a state function, but it can be identified with {\em heat
exchange rate}.  Then in the stationary state, when
$\Delta S = 0$, there is 
positive entropy production rate that is balanced with 
the heat dissipation.  Such a state is called a 
{\em nonequilibrium steady state} \cite{jqq}.

\subsection{Shannon-Khinchin entropy and Entropy Production for MC}

For the entire path, a more rigorous construction of MC is to consider a measurable dynamical system $(\Omega', \mathcal{F}', \mathbb{P}', \theta(t))$, where $\Omega'=\mathscr{S}^\mathbb{Z}$,  $\theta$ is again the shift map. The probability measure $\mathbb{P}'$ is defined on the cylinder set $[i_0, i_{1}, \dots, i_{t}]$,
 \begin{align}
  \mathbb{P}'([i_0, i_{1}, \dots, i_{t}])=\vp_{i_0}(0)M_{i_0i_{1}}\dots M_{i_{t-1}i_{t}}.
\end{align}
Since $\vp(0)$ may be not necessarily its stationary distribution, $\mathbb{P}'$ is not $\theta$-invariant.  
But marginalizing states at previous steps $0, 1, \dots, t-1$, it gives the probability at step $t$, i.e, $\sum_{i_0, \dots, i_{t-1}} \mathbb{P}'([i_0, i_{1}, \dots, i_{t}])=\vp_{i_{t}}(t)$, where $\vp(t)=\vp(0)M^t$.
The stochastic process $X_t$ is defined as $X_t(\omega)\triangleq\theta(t)\omega=\omega_t$. 

For the MC, applying (\ref{shannon}) to the finite time 
distribution of MC restricted to $\sigma$-field $\mathcal{F'}_0^{t}$, $S(\mathbb{P'}_{[0,t]})$, where $\mathcal{F'}_0^{t}=\sigma(X_s: 0\le s \le t)$, is called Shannon-Khinchin entropy \cite{khinchin}
\begin{align}
 H_{SK}(\mathbb{P'}_{[0,t]})=-\sum_{i_0,\dots, i_{t}} \vp_{i_0}(0)M_{i_0i_{1}}\dots M_{i_{t-1}i_{t}}\log\Big(\vp_{i_0}(0)M_{i_0i_{1}}\dots M_{i_{t-1}i_{t}}\Big).
\end{align}

It relates to the metric entropy of the MC via 
\begin{align}
\label{mc-m-entropy}
h_{\text{MC}}&=\lim_{t\rightarrow+\infty}\frac{ H_{SK}(\mathbb{P'}_{[0,t]})}{t}=-\sum_{i,j\in\mathscr{S}}\pi_iM_{ij}\log(M_{ij}),
\end{align}
which is a property for stationary MC. 
The Shannon-Khinchin entropy of a step is $\ \dbar H_{SK}(t)\equiv H_{SK}(\mathbb{P'}_{[0,t+1]})-H_{SK}(\mathbb{P'}_{[0,t]})$. Here one can show the asymptotic limit
of $\ \dbar H_{SK}(t)$ is the metric entropy $h_{\text{MC}}$
as $t\to +\infty$. The metric entropy quantifies the average randomness (or information) generated per step in an MC.

More detailed analysis, including the relation between metric entropy of an RDS and its corresponding RDS, can be found in \cite{ywq}. Briefly, first, the set of Markov transition matrices forms the convex hull with deterministic transition matrices as its vertices. Second, 
 metric entropy of RDS  corresponding to the MC has the upper bound, i.e, $h_{RDS}\le -\sum_{i,j\in\mathscr{S}} M_{ij}\log(M_{ij})$ and such representation is uniquely attainable since a strictly concave function over a convex hull has the unique maximum. Moreover, there is an explicit expression for such maximum entropy RDS, i.e, $Q(P_{i_1,i_2,\dots, i_n})=M_{1i_1}M_{2i_2}\dots M_{ni_n}$, where $P_{i_1,i_2,\dots, i_n}$ corresponds to the deterministic map $1\rightarrow i_1, 2\rightarrow i_2, \dots, n\rightarrow i_n$.

From now on, we consider the transition probability matrix $M$ satisfies the condition $M_{ij}>0 \leftrightarrow M_{ji}>0$ for any $i, j\in \mathscr{S}$, then it is possible to define the relative entropy of the distribution of the process with respect to its time reversal restricted to $\sigma$-field $\mathcal{F'}_0^{t}$. The time-reversed process $X^-$ is defined as follows, 
\begin{align} 
X^-_s(\omega)=X_{t-s}(\omega), \ \forall s\in[0, t].
\end{align}
So $X^-$ is $\mathcal{F'}_0^{t}$ measurable. The time-reversed process is also called {\it adjoint process} of the MC. For the sample sequence $i_0, i_{1}, \dots, i_{t}$, the time-reversed process gives $i_{t}, i_{t-1}, \dots, i_0$. Define $\mathbb{P'}^-$ as the probability measure for the time-reversed process $X^-_{s}(\omega)$. The probability measure for the time-reversed process $\mathbb{P'}^-$ on this cylinder set is 
\begin{align}
\mathbb{P'}^-([i_{t}, i_{t-1}, \dots, i_{s}])=\vp_{i_{s}}(s)M_{i_{s}i_{s+1}}\dots M_{i_{t-1}i_t}, \ \text{for any $s\in[0,t-1]$}
\end{align}
Here we assume $\vp_i(s)>0$ for all $i$ and $s\in[0,t-1]$. Since the process is non-stationary, it is necessary to indicate the initial time $s$. 

\begin{proposition}
The time-reversed process of a Markov chain $(M, \vpi)$ is Markovian. Moreover, the transition matrix of the time-reversed process is \begin{align}
M^-_{ij}(t)=\frac{\vp_{j}(t-1) M_{ji}}{\vp_{i}(t)},\ i,j \in  \mathscr{S} \ \text{for any $t$.}
\end{align}
 \end{proposition}
 \begin{proof}
The Markovian property means $P(X^-_{t+1}=i_{s-1}| \mathcal{F'}_s^t)=P(X^-_{t+1}=i_{s-1}| X^-_t)$ for any $t\ge s $. The left-hand-side is 
 \begin{align*}
 P(X^-_{t+1}=i_{s-1}| \mathcal{F'}_s^t)=\frac{\mathbb{P'}^-([i_{t}, \dots, i_{s}, i_{s-1}])}{ \mathbb{P'}^-([i_{t}, \dots, i_{s}])}=\frac{\vp_{i_{s-1}}(s-1)M_{i_{s-1}i_s}}{\vp_{i_s}(s)}.
 \end{align*}
The right-hand-side is 
 \begin{align*}
 P(X^-_{t+1}=i_{s-1}| X^-_t)=\frac{\mathbb{P'}^-([i_{s}, i_{s-1}])}{\mathbb{P'}^-([i_{s}])}=\frac{\vp_{i_{s-1}}(s-1)M_{i_{s-1}i_s}}{\vp_{i_s}(s)}.
 \end{align*}

 Moreover
 \begin{align*}
\mathbb{P'}^-([i_{t}, i_{t-1}, \dots, i_{s}])&=\vp_{i_s}(s) M_{i_{s}i_{s+1}}\dots M_{i_{t-1}i_t} \\
&=\vp_{i_t}(t)\underbrace{\Big(\frac{\vp_{i_{t-1}}(t-1)M_{i_{t-1} i_{t}}}{\vp_{i_{t}}(t)}\Big)}_{M^-_{i_ti_{t-1}}(t)}\dots \underbrace{\Big(\frac{\vp_{i_s}(s)M_{i_s i_{s+1}}}{\vp_{i_{s+1}}(s+1)}\Big)}_{M^-_{i_{s+1}i_s}(s+1)}.
\end{align*}
the probability measure can be rewritten as the Markov measure with the transition matrix at time $t$, $M^-_{ij}(t)=\frac{\vp_{j}(t-1) M_{ji}}{\vp_{i}(t)}$.
 \end{proof}

 So this time-reversed process of a Markov chain is time-inhomogeneous Markov process. One can check $M^-_{ij}(t)$ indeed is a Markov matrix.
  In particular, for the stationary MC, the transition matrix for the time-reversed process is $M^-_{ij}=\frac{\vpi_j M_{ji}}{\vpi_i}$. The stationary distribution of the time-reversed process is also $\vpi$. 
 
 Since $\vp_i(0)>0$, $\mathbb{P}'_{[0, t]}$ is absolutely continuously with respect to $\mathbb{P}'^-_{[0,t]}$.
 Then the relative entropy of the measure of Markov chain with respect to the measure of time-reversed process, $H(\mathbb{P'}_{[0,t]}, \mathbb{P'}^-_{[0,t]})$ is 
\begin{align}
\label{rel-path-ent}
H\left(\mathbb{P'}_{[0,t]}, \mathbb{P'}^-_{[0,t]}\right)=\sum_{i_0, \dots, i_{t}\in\mathscr{S}}\vp_{i_0}(0)M_{i_0i_{1}}\dots M_{i_{t-1}i_{t}}\log\Big(\frac{\vp_{i_0}(0)M_{i_0i_{1}}\dots M_{i_{t-1}i_{t}}}{\vp_{i_{t}}(0)M_{i_{t}i_{t-1}}\dots M_{i_{1}i_0}}\Big).
\end{align}
Similarly, it relates to the stationary entropy production rate of MC via 
\begin{align}
\label{ep-def}
e_p=\lim_{t\rightarrow +\infty}\frac{1}{t}H\left(\mathbb{P'}_{[0,t]}, \mathbb{P'}^-_{[0,t]}\right).
\end{align}
So the $e_p$ is intuitively understood as the asymptotic average of entropy produced per step with respect to its time-reversed probability. It is also a property for a stationary MC. It has the following explicit expression \cite{jqq}.
\begin{theorem} 
\begin{align}\label{epr}
e_p=\sum_{i,j \in  \mathscr{S}} \vpi_iM_{ij}\log\Big(\frac{M_{ij}}{M_{ji}}\Big).
\end{align}
\end{theorem}
\begin{proof}
From the Definition \ref{ep-def}
\begin{align*}
e_p&=\lim_{t\rightarrow +\infty}\frac{1}{t}H\left(\mathbb{P'}_{[0,t]}, \mathbb{P'}^-_{[0,t]}\right) \\
&=\lim_{t\rightarrow+\infty}\frac{1}{t}\left\{\sum_{i\in\mathscr{S}}\big(\vp_{i}(0)-\vp_{i}(t)\big)\log \vp_{i}(0)+\sum_{s=0}^{t-1}\sum_{i,j\in  \mathscr{S}}\vp_{i}(s)M_{ij}\log\Big(\frac{M_{ij}}{M_{ji}}\Big)\right\} \\
&=\sum_{i,j\in  \mathscr{S}}\vpi_i M_{ij}\log\Big(\frac{M_{ij}}{M_{ji}}\Big).
\end{align*}

\end{proof}

	One in fact has a result stronger than (\ref{ep-def}):  The
relative entropy in (\ref{rel-path-ent}) of a step is 
\begin{eqnarray}
 \dbar H\Big(\mathbb{P}'_{[0,t]}, \mathbb{P}'^-_{[0,t]}\Big) 
   &\equiv& H\Big(\mathbb{P}'_{[0,t+1]}, \mathbb{P}'^-_{[0,t+1]}\Big)-H\Big(\mathbb{P}'_{[0,t]}, \mathbb{P}'^-_{[0,t]}\Big)
\nonumber\\
	&=&\sum_{ij}\vp_i(t)M_{ij}\log\left(
          \frac{\vp_i(0)M_{ij}}{\vp_j(0)M_{ji}}\right).
\label{eq-29}
\end{eqnarray}
One can show that the asymptotic relative entropy
of a step in (\ref{eq-29}) is also the $e_p$. 

Note that for a stationary MC, the entropy production rate is exactly the time-averaged relative entropy, i.e, $e_p=\frac{1}{t}H\big(\mathbb{P'}_{[0,t]}, \mathbb{P'}^-_{[0,t]}\big)$ for any $t>0$.
There are many other equivalent expressions for the entropy production rate. For instance, 
\[
  e_p=\sum_{i,j} \vpi_iM_{ij}\log \Big(\frac{\vpi_i M_{ij}}{\vpi_j M_{ji}}\Big), \  \text{ or }  \
e_p=\sum_{i,j} \vpi_iM_{ij}\log \Big(\frac{M_{ij}}{M^-_{ij}}\Big),
\]
or
\[
     e_p=\frac{1}{2}\sum_{i,j} (\vpi_iM_{ij}-\vpi_jM_{ji})\log\Big(\frac{\vpi_iM_{ij}}{\vpi_jM_{ji}}\Big).
\]
So the entropy production rate $e_p=0$ if and only if the MC is detailed balance, $\vpi_iM_{ij}=\vpi_jM_{ji}$.

	As we have discussed earlier, both 
$\ \dbar H_{SK}\big(\mathbb{P}'_{[0,t]}\big)$ and 
$\ \dbar H\big(\mathbb{P}'_{[0,t]}, \mathbb{P}'^-_{[0,t]}\big)$ 
belong to the same class of stochastic quantitie. 
As their asymptotic limits, the metric entropy $h_{MC}$ 
and the entropy production rate $e_p$ characterize the average randomness  generated in the dynamic stepping from $t$ to $t+1$ and average dynamic asymmetry with 
respect to time reversal, respectively. 
In the same class, there is another 
non-negative quantity:
\begin{equation}
 \sum_{i,j}
        \vp_i(t-1)M_{ij}\ln\left(\frac{\vp_i(t-1)M_{ij}}{
        \vp_j(t)M^-_{ji} }\right) = -\Delta H\Big(\vp(t-1),\vpi \Big),
\end{equation}
in which $M_{ji}^-=\vpi_iM_{ij}/\vpi_j$ is the transition
matrix for the time-reversed process. 
The right-hand-side is actually the
$\Delta H\big(\vp(t-1),\vpi\big)$ in Theorem \ref{theo-voigt}.
It signifies non-stationarity,
e.g., asymmetry with respect to translation in time \cite{qian-jmp}. The term inside the logarithm $\frac{\vp_i(t-1)M_{ij}}{
        \vp_j(t)M^-_{ji} }$ goes to 0 and the term $\vp(t)$ goes $\vpi$ as $t\rightarrow +\infty$. Therefore, the asymptotic limit $ \Delta H\big(\vp(t-1),\vpi \big)$ goes to 0.

\subsection{Cycle Distributions and RDS}
\label{sec-4.3}

Besides the expression of the entropy production rate in terms of the transition matrix $M$ and its corresponding stationary distribution $\vpi$, a different representation can be given in terms of a collection of cycles $\mathcal{C}$ and weights $\{w_c: c\in \mathcal{C}\}$ on these cycles. These are regarded as cycle distributions of the MC \cite{kalpazidou2007}. In addition, there is an associated graph-based diagram method to compute the weight $w_c$ which was first discovered by T. L. Hill \cite{hill2004} and proved by Qians \cite{jqq}. In the setting of maximum entropy RDS, this graphical method can be further formulated as a function on the probability coefficient of the deterministic map with the single attractor.
The entropy production rate $e_p$ of the MC can be expressed in terms of the cycle weights as,
\begin{align} \label{epr-cycle}
e_p=\sum_{c\in\mathcal{C}}w_c\log\frac{w_c}{\ w_{c_-}}.
\end{align}
Here $c_-$ denotes the reversed cycle of $c$. 
The previous proof  given in \cite{jqq} is quite involved.  Here 
we will give a shorter, combinatorial proof for Eq. (\ref{epr-cycle}). 

We start with discussing a related graphical method to solve the invariant distribution for the MC, i.e. to solve $\vpi$ for $\vpi M=\vpi$. The MC can be viewed as a directed graph with the transition probability as edge weight. Firstly, construct the complete set of spanning directed rooted {\em trees} which all edges are directed toward the root. A tree has the maximum possible edges without forming any loops. So for a tree with $n$ nodes has $n-1$ edges and each node, except the root, has exactly one outgoing edge. If one views the directed graph as a discrete map for a dynamical system, the directed rooted tree gives a single fixed point. Second, assign the weight of previous directed rooted tree $T$ as the product of its edge weights, $e(T)$. In fact, this weight connects with the coefficient of the corresponding map $\alpha$ (with root goes to itself) under maximum entropy RDS as follows, $Q(\alpha)=e(T)M_{ii}$, where $i$ is the root state. Finally, the weight of a set of graphs is the sum of their weights. Then the invariant distribution $\vpi$ can be expressed by $e(T)$ \cite{King1956, hill2004, Caplan}.
\begin{theorem} \label{Hill-theorem}
 The invariant distribution for the irreducible and aperiodic MC is given by 
\begin{align}\label{Hill-pi}
\vpi_i=\frac{e(\mathcal{T}_i)}{\Sigma}, \ i=1,\dots, n.
\end{align}
where $\mathcal{T}_i$ is the set of directed rooted trees whose root is state $i$ and the normalization factor $\Sigma=\sum_{i=1}^n e(\mathcal{T}_i)$.
\end{theorem}
The original proof was based on Cramer's rule but Hill discovered an elegant proof which is included in Appendix \ref{appendix-A}. The connection between two proofs is the matrix-tree theorem and is discussed in Appendix \ref{appendix-B}. From the proof, it gives the insight for the following proposition to connect the RDS picture.

We call the attractor of the deterministic map $\alpha$, $\mathcal{A}(\alpha)$, is {\it single} if the attractor is either exactly one fixed point or one limit cycle.
The size of the single attractor $\|\mathcal{A}(\alpha)\|$ is the period of the limit cycle or just one for the fixed point. If the attractor is not single, denote $\mathcal{A}(\alpha)=\emptyset$ and the size of it is 0. 

Let $\mathcal{G}_i$ be the set of directed graphs that has exactly one limit cycle and $i$ is contained in that cycle. Unlike the set of directed rooted trees $\mathcal{T}_i$,  for any $G\in \mathcal{G}_i$, $e(G)$ is exactly the probability of the corresponding deterministic map under maximum entropy RDS. 
  So the set of directed graph corresponds to the deterministic maps with single attractor is $\cup_i \{\mathcal{G}_i,\mathcal{T}_i\}$. For convenience, denote $\mathcal{A}(G)$ as the single attractor of this corresponding deterministic map for $G\in \cup_i \{\mathcal{G}_i,\mathcal{T}_i\}$.

\begin{proposition}\label{sigma-rds}
The normalization factor $\Sigma$ is equal to the expected size of the single attractor of the deterministic map under maximum entropy RDS, i.e, 
\begin{align}
\sum_{i=1}^n e(\mathcal{T}_i)=\mathbb{E}^Q(\|\mathcal{A}(\alpha)\|).
\end{align}
\end{proposition}

\begin{proof}

From appendix \ref{appendix-A}, we know $\sum_{\substack {j=1 \\ j\ne i}}^nM_{ij}e(\mathcal{T}_i)=e(\mathcal{G}_i)$. So 
\begin{align*}
\sum_{i=1}^n e(\mathcal{T}_i)=\sum_{i=1}^n \sum_{j=1}^n M_{ij}e(\mathcal{T}_i)=\sum_{i=1}^n \Big(e(\mathcal{G}_i)+M_{ii}e(\mathcal{T}_i)\Big).
\end{align*}
 $\sum_{i=1}^n M_{ii}e(\mathcal{T}_i)$ is the probability of the deterministic maps whose attractor is exactly one fixed point. It can be rewritten as $\sum_{T\in \cup_i \mathcal{T}_i} M_{\mathcal{A}(T)\mathcal{A}(T)}e(T) $.  Here $\mathcal{A}(T)$ is the root state. 
 For any $G\in \mathcal{G}_i$, $e(G)$ is counted $\|\mathcal{A}(G)\|$ times in $\sum_{i=1}^n e(\mathcal{G}_i)$. So it can rewritten as $\sum_{i=1}^n e(\mathcal{G}_i)= \sum_{G\in \cup_i \mathcal{G}_i}\|\mathcal{A}(G)\| e(G)$ and can be interpreted as the expected size of the single limit cycle. In total, it is the expected size of the single attractor of the deterministic map under maximum entropy RDS. 
 
\end{proof}


The MC will generate an infinite sequence of cycles almost surely. The set of cycles $\mathcal{C}$ contains all possible directed cycles along almost all sample paths of $X_t$ and the weight $w_c$ is the mean occurrences of the directed cycle $c$ for almost all sample paths.  The main reason to introduce this graphic method is it can be extended to find the cycle weight $w_c$. Here we present the  method, the basic idea of the proof and the insight in RDS picture. 

Start with the sample sequence $\omega$, decompose cycles along the sequence by discarding the cycles formed at step $t$ and keep the track of the remaining states in the sequence. The remaining sequence is called {\it derived chain} $\eta_t(\omega)$ and $w_{c,t}(\omega)$ is the {\it number of occurrences }of the cycle $c$ up to step $t$. The average rate of occurrence up to time $t$ is $\frac{w_{c,t}(\omega)}{t}$. The weight also enjoys the ergodicity, i.e. $\lim_{t\rightarrow+\infty} \frac{w_{c,t}(\omega)}{t}=\omega_c$, almost surely.
 The precise definition and the proof of the ergodicity is given in \cite{jqq}. Here is one example to illustrate the idea. If we give the trajectory $(2, 1, 3, 1, 3,1, 1, 2, 2, \dots)$ of $X_t(\omega)$, where the transition matrix is 
\[
       M=\begin{bmatrix} M_{11} & M_{12} & M_{13} \\ M_{21} & M_{22} & M_{23}\\ M_{31} &M_{32} &M_{33} \end{bmatrix}
\]  
The derived chain dynamics and cycles are counted as in Table \ref{table-1}. 

\begin{table}
\begin{tabular}{cccccccccc}

$t$ & 0&1 &2&3&4&5 &6&7&8 \\[3pt]
$X_t(\omega)$ &2&1&3&1&3&1 &1&2&2\\[3pt]
$\eta_t(\omega)$ & [2]& [2,1]&[2,1,3] &[2,1]&[2,1,3]&[2,1]&[2,1]&[2]&[2]\\[3pt]
$cycles$ & &&&(1,3)& &(3,1) &(1) &(1,2) &(2)\\
\smallskip
\end{tabular}
\caption{The derived chain $\eta_t$ and the cycles formed for this sample trajectory $X_t$ \cite{jqq}.}
\label{table-1}
\end{table}

Note the self-loop $2\rightarrow 2$ here is regarded as a degenerated cycle. Without loss of generality, we include such self-loops to $\mathcal{C}$. Cycles are recorded by the ordered sequence $c=(i_1, \dots, i_t)$ with  
$i_s$ are distinct for $1\le s\le t$. The reverse of the cycle $(i_1, \dots, i_t)$ is defined as $c_-=(i_t, \dots, i_1)$. So $(1,3)$ and its reverse $(3,1)$ are the same cycle, but $(2,1,3)$ and its reverse $(3,1,2)$ are different. These cycles are the single attractors of deterministic maps, $\mathcal{C}=\mathcal{A}(\cup_i \{\mathcal{G}_i,\mathcal{T}_i\})$.

In fact, the dynamic of the derived chain itself also deserves a look. $\eta_t$ follows another MC with state space as all possible distinct ordered sequences. From the example, one can see that the first element in the derived chain state is always the initial state. Here the MC for the derived chain has 15 states. The transition matrix $\widetilde{M}$ is reducible with 3 classes based on the initial state. For the class starting with state 2, the set of all possible states are $[ \mathscr{S}]_{2}=\{[2], [2,1],[2,3], [2,1,3],[2,3,1]\}$. The submatrix of $\widetilde{M}$ on $[ \mathscr{S}]_{2}$ is
\begin{align} \label{M2}
\widetilde{M}^{[2]}=\begin{blockarray}{cccccc}
&$[2]$ &$[2,1]$ &$[2,3]$ &$[2,1,3]$& $[2,3,1]$ \\
\begin{block}{c(ccccc)}
$[2]$ & M_{22} & M_{21} & M_{23} && \\
  $[2,1]$ & M_{12} & M_{11} &  &M_{13}&&\\
  $[2,3]$ &M_{32} &  &M_{33} &&M_{31}\\
  $[2,1,3]$ &M_{32}& M_{31} & & M_{33} & \\
  $[2,3,1]$ &M_{12}&& M_{13} &&M_{11} \\
\end{block}
\end{blockarray}
\end{align}
The directed graph corresponds to $\widetilde{M}^{[2]}$ is shown in Fig. \ref{figure-ex1}. There is another example of derived chain dynamics in Fig. \ref{figure-ex2}. The transition matrix of the MC is 
\begin{align}
M=\begin{bmatrix} 0 & M_{12} & 0 & M_{14} \\ M_{21} & 0 & M_{23}& M_{24} \\ 0 &M_{32} &0 &M_{34} \\ M_{41} & M_{42} & M_{43} & 0 \end{bmatrix}
\end{align}
Here are the properties of the derived chain dynamics on the connectivity and cycle formation. Reader can refer Fig. \ref{figure-ex1} and Fig. \ref{figure-ex2} for more intuitive ideas. Denote the size of the derived chain state $[i_1, \dots, i_t]$ as $t$. Then the largest possible size of the derived chain state is $n$. 

First, the number of states in derived chain is much more than the original MC and not every state is connected in derived chain dynamics even when the original MC is completely connected. The derived chain state with size of $t$ can only go to the state with size of $s\le \min (t+1, n)$. There is no connection between states with the same size except with itself. 
Each state with the last element $i_t$ implies there is a path in the original MC from $i_1$ to $i_t$ without forming a cycle. So the number of such paths is the same as number of states with the last element $i_t$. Furthermore, Each states with size $n$ gives a Hamiltonian path in the original MC from $i_1$ which is the path from $i_1$ visits each state exactly once. 
Since the last element of the derived chain state is the current state in the original MC, the weight of each edge $[i_1, \dots, i_t][i_1, \dots, i_s]$ if exists, is the transition probability $M_{i_ti_s}$.  
 Therefore the invariant distribution for the derived chain state $\vPi^{i_1}([i_1, \dots, i_t])$ satisfies the following equation, 
\begin{align}\nonumber
\vPi^{i_1}([i_1, \dots, i_t])&=\vPi^{i_1}([i_1, \dots, i_{t-1}])M_{i_{t-1}i_t}+\vPi^{i_1}([i_1, \dots, i_t])M_{i_ti_t} \\  \label{JQQ-Pi-eq}
&+\sum_{j_1,\dots,j_r} \vPi^{i_1}([i_1, \dots, i_t, j_1,\dots,j_r])M_{j_ri_t},
\end{align}
where $j_1,\dots, j_r$ have no common elements with $i_1,\dots, i_{t}$.
In addition, since the original MC is ergodic, the derived chain is also ergodic on each irreducible class $[\mathscr{S}]_{i_1}$.  Then the invariant distribution of the derived chain should satisfy the equation (\ref{JQQ-Pi-eq}) and is unique up to the normalization. The detailed proof is in Theorem \ref{JQQ-Pi-theo}.

 Second, in the original MC the current state cannot tell which cycle could be possibly formed at next step, but the derived chain state can. 
The derived chain has the past history on the path from $i_1$ to the current state of the original MC after removing cycles. Each time the derived chain state $[i_1, \dots, i_t]$ goes to the state with the same or less size $[i_1, \dots, i_s], s\le t$, a cycle $(i_s, i_{s+1}, \dots, i_t)$ in the original MC sequence is formed. In particular, any cycle contained with $i_1$, i.e, $(i_1, \dots, i_t)$, can only be formed by the derived chain state $[i_1, i_2, \dots, i_t]$ going to $[i_1]$.

From these properties, one can pick the suitable class to calculate the weight of cycles. In particular, for the cycle $c=(i_1, i_2, \dots, i_t)$, it is much convenient to use the class $[ \mathscr{S}]_{i_1}$, i.e, the initial state of the original MC is $i_1$. 
 Hence the average rate of the occurrence for this cycle is same as the average frequency of the pair $[i_1, \dots, i_t][i_1]$ in the derived chain dynamics. Since the average frequency of state $[i_1, \dots, i_t]$ goes to $\vPi^{i_1}([i_1, i_2, \dots, i_t])$, the average frequency of the pair $[i_1, \dots, i_t][i_1]$ is 
 $w_c=\vPi^{i_1}([i_1, i_2, \dots, i_t])\cdot M_{i_ti_1}$.

\begin{proposition}
The directed graph $G_2$ corresponds to the derived chain on $[ \mathscr{S}]_{i_1}$, is homomorphic to the directed graph $G_1$ to the MC on $ \mathscr{S}$.
\end{proposition}
\begin{proof}
Consider the function from $[ \mathscr{S}]_{i_1}$ to $ \mathscr{S}$, $f([i_1, \dots, i_t])=i_t$. For any edge  $[i_1, \dots, i_t][i_1, \dots, i_s]\in E(G_2)$, the weight of the edge is $M_{i_ti_s}$ and it means there is an edge from $i_t$ to $i_s$ in $G_1$. So $f([i_1, \dots, i_t])f([i_1, \dots, i_s])\in E(G_1)$. 
\end{proof}
The Fig. \ref{figure-ex1} and Fig. \ref{figure-ex2} demonstrate the homomorphism of the directed graph $G_2$ (left) to $G_1$ (right). The homomorphism $f$ induces the equivalent class on $[ \mathscr{S}]_{i_1}$ by using the pre-image $f^{-1}(i_t)$ for any $i_t\in  \mathscr{S}$. These equivalent classes are colored in the same color in these two examples. In fact, the invariant measures on both dynamics are also connected by the pre-image of the homomorphism as shown in the following theorem.

\begin{figure}
\includegraphics[scale=0.55]{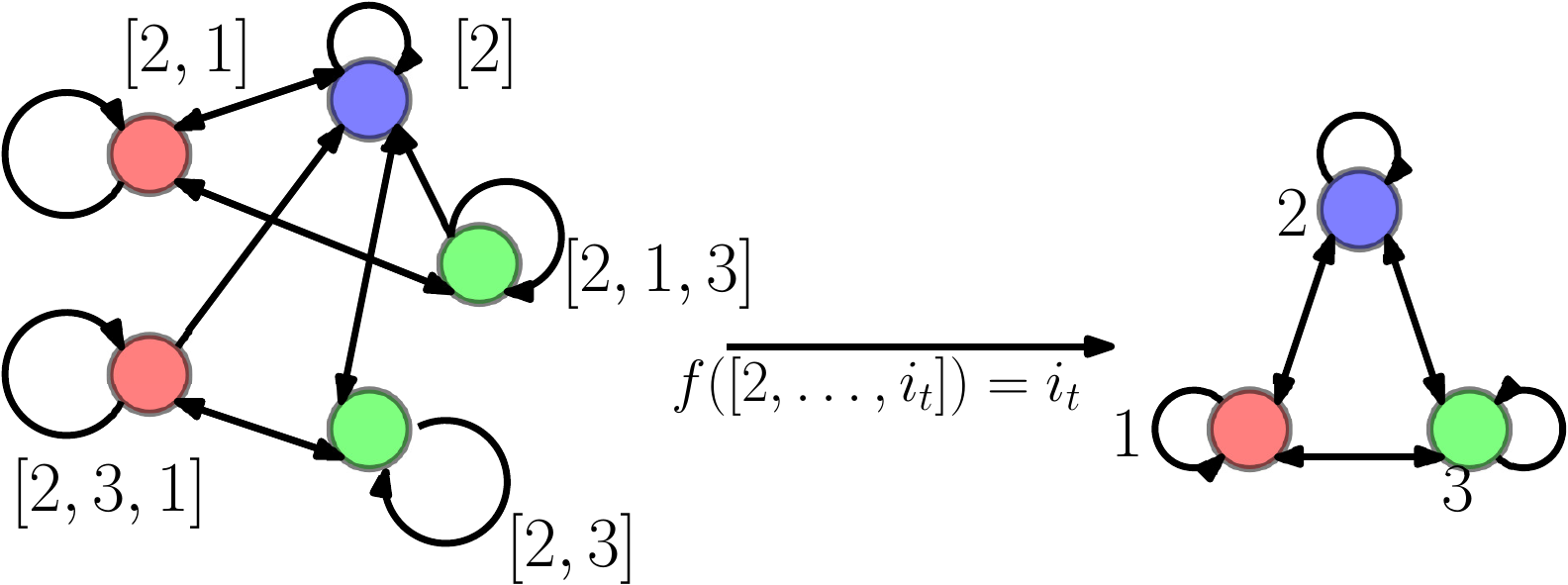}
\caption{ A 3-state completely connected MC. The transition matrix of derived chain dynamics is in (\ref{M2})}
\label{figure-ex1}
\end{figure}

\begin{figure}
\includegraphics[scale=0.5]{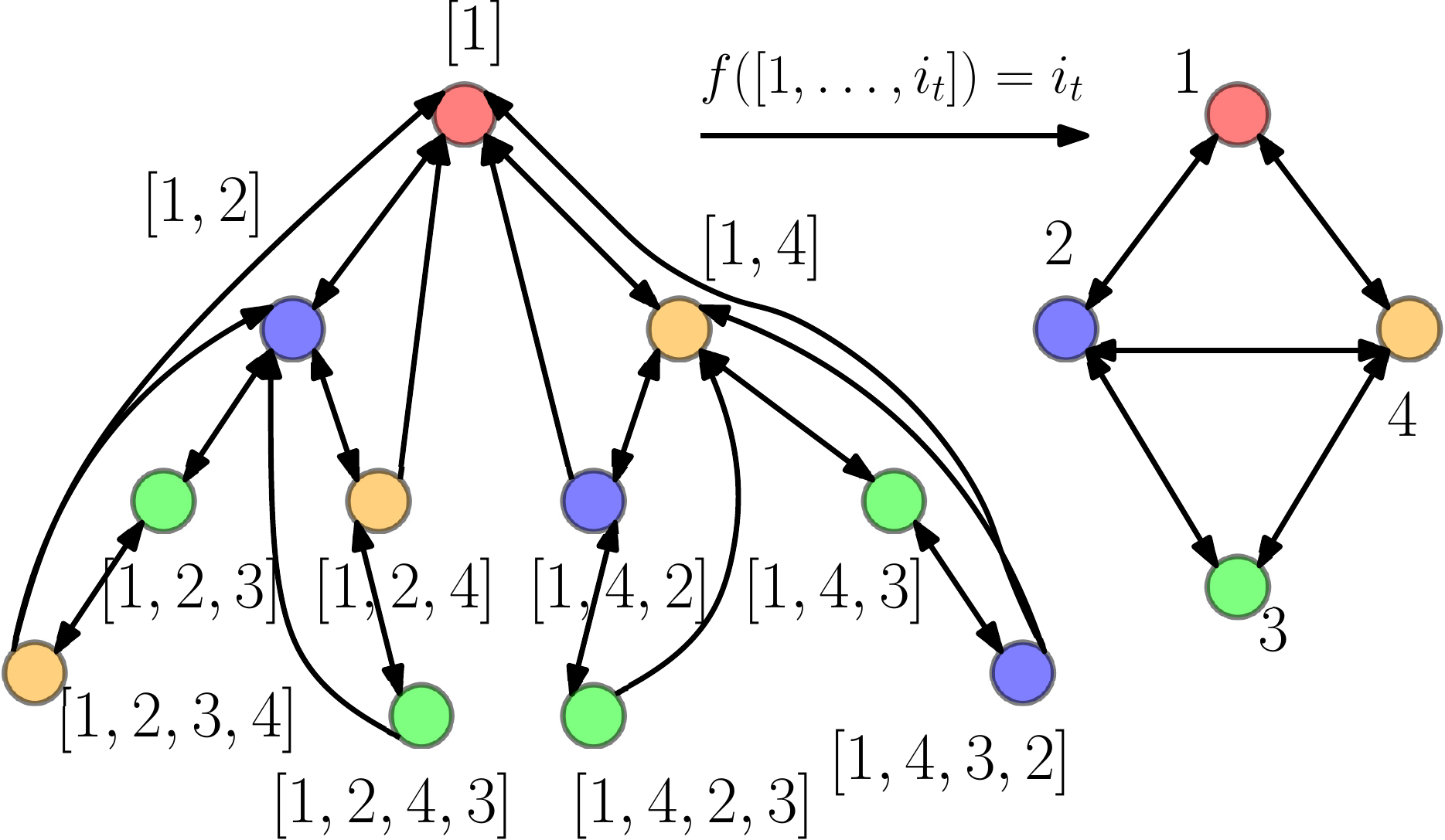}
\caption{ A 4-state MC with initial state 1. }
\label{figure-ex2}
\end{figure}

   A directed rooted tree is denoted as $T_{[i_1, \dots, i_t]}$ if it has root $i_t$ and has edges $i_1i_2, \dots i_{t-1}i_t$. By convention, if $t=1$, $T_{[i_1]}=T_{i_1}$. And $\mathcal{T}_{[i_1, \dots, i_t]}$ is the set of all such directed rooted  trees $T_{[i_1,\dots, i_t]}$. 

\begin{theorem} \label{JQQ-Pi-theo}
The invariant distribution for the derived chain dynamics is given by
\begin{align}\label{JQQ-Pi}
\vPi^{i_1}([i_1, \dots, i_t])=\frac{e(\mathcal{T}_{[i_1,\dots, i_t]})}{\Sigma}, \ [i_1, \dots, i_t]\in [ \mathscr{S}]_{i_1}. 
\end{align}
where the normalization factor is $\Sigma=\sum_{i=1}^n e(\mathcal{T}_i)$. Moreover, $\vPi^{i_1}(f^{-1}(i_t))=\vpi(i_t)$.
\end{theorem}
Note: 
\begin{align} \label{eT-express}
e(\mathcal{T}_{[i_1,\dots, i_t]})=M_{i_1i_2}\dots M_{i_{t-1}i_t}e(\mathcal{T}_{i_1,\dots, i_t}),
\end{align}
 where $\mathcal{T}_{i_1,\dots, i_t}$ is the set of directed forests whose roots are $i_1, \dots, i_t$. 
\begin{proof}

The first part is to show the invariant distribution of the derived chain and the weights in (\ref{JQQ-Pi}) of the set of directed rooted trees satisfy the same system of linear equations (\ref{JQQ-Pi-eq}) and the solution is unique up to a normalization, i.e, $\vPi^{i_1}([i_1, \dots, i_t])\propto e(\mathcal{T}_{[i_1,\dots, i_t]})$. The technique is similar to the one in appendix \ref{appendix-A}. Based on Eq. (\ref{JQQ-Pi-eq}), it is equivalent with solving the following equation,  
\begin{eqnarray}
\label{JQQ-Pi-eq-mod}
\sum_{\substack {i_s=1\\i_s\ne i_t}}^{n} \vPi^{i_1}([i_1, \dots, i_t])M_{i_ti_s} &=&  \vPi^{i_1}([i_1, \dots, i_{t-1}])M_{i_{t-1}i_t}
\nonumber\\
 &+&\sum_{j_1,\dots,j_r} \vPi^{i_1}([i_1, \dots, i_t, j_1,\dots,j_r])M_{j_ri_t},
\end{eqnarray}
where $j_1,\dots, j_r$ have no common elements with $i_1,\dots, i_{t}$. 
Let $\mathcal{G}_{[i_1, \dots, i_t]}$ be the set of directed graphs that have exactly one limit cycle with $i_t$ contained and the graphs have edges $i_1i_2, \dots, i_{t-1}i_t$. 
\begin{enumerate}

\item $\mathcal{T}_{[i_1,\dots, i_t]}+i_ti_s\subset \mathcal{G}_{[i_1, \dots, i_t]}$  If a directed rooted tree $T\in \mathcal{T}_{[i_1,\dots, i_t]}$, adding an edge $i_ti_s$ will create an element $G\in \mathcal{G}_{[i_1, \dots, i_t]}$, where $i_s\ne i_t$.

\item $\mathcal{G}_{[i_1, \dots, i_t]}- i_ti_s \subset \mathcal{T}_{[i_1,\dots, i_t]}$ If a directed graph $G\in \mathcal{G}_{[i_1, \dots, i_t]}$, deleting the edge $i_ti_s$ will create an element $T\in \mathcal{T}_{[i_1,\dots, i_t]}$, where $i_s\ne i_t$.

\item $\mathcal{T}_{[i_1,\dots, i_{t-1}]}+ i_{t-1}i_t \subset  \mathcal{G}_{[i_1, \dots, i_t]}$ If a directed rooted tree $T\in \mathcal{T}_{[i_1,\dots, i_{t-1}]}$, adding an edge $i_{t-1}i_t$ will create an element $G\in \mathcal{G}_{[i_1, \dots, i_t]}$ because $i_t$ has a path to $i_{t-1}$ in the tree and now with the edge $i_{t-1}i_t$, it must have a limit cycle with $i_t$ contained. 

\item $ \mathcal{T}_{[i_1, \dots, i_t, j_1, \dots, j_r]}+j_ri_t\subset  \mathcal{G}_{[i_1, \dots, i_t]}$ If a directed rooted tree $T\in \mathcal{T}_{[i_1, \dots, i_t, j_1, \dots, j_r]}$, adding an edge $j_r i_t$ will create a limit cycle $(i_t, j_1, \dots, j_r)$, which is an element in $\mathcal{G}_{[i_1, \dots, i_t]}$. 

\item 
 If a directed graph $G\in \mathcal{G}_{[i_1, \dots, i_t]}$ has the limit cycle $(i_t, j_1, \dots, j_r)$, $r\ge 1$, where $j_1, \dots, j_r$ have no common elements with $i_1, \dots, i_{t-1}$, deleting the edge $j_r i_t$ will create an element $T\in \mathcal{T}_{[i_1, \dots, i_t, j_1, \dots, j_r]}$ since the root is $j_r$ and apart from the path from $i_1$ to $i_t$, there is another path from $i_t$ to $j_r$: $i_tj_1, \dots j_{r-1}j_r$,  

\item 
If a directed graph $G\in\mathcal{G}_{[i_1, \dots, i_t]} $ has the limit cycle $(i_s,\dots, i_t, j_1, \dots, j_r)$, $r\ge 1$, where $1\le s< t$ and where $j_1, \dots, j_r$ have no common elements with $i_1, \dots, i_t$, deleting the edge $i_{t-1}i_t$ will create an element $T\in \mathcal{T}_{[i_1,\dots, i_{t-1}]}$ because the root is $i_{t-1}$.

\end{enumerate}
From 1 and 2, we showed $\mathcal{T}_{[i_1,\dots, i_t]}+i_ti_s= \mathcal{G}_{[i_1, \dots, i_t]}$, then the RHS of Eq. (\ref{JQQ-Pi-eq-mod}) is 
$$\text{RHS}=\sum_{\substack {i_s=1\\i_s\ne i_t}}^{n}e(\mathcal{T}_{[i_1, \dots, i_t]})M_{i_ti_s}=e(\mathcal{G}_{[i_1, \dots, i_t]}).$$
 From 3-6, we showed $\{\mathcal{T}_{[i_1,\dots, i_{t-1}]}+ i_{t-1}i_t\} \cup \{\mathcal{T}_{[i_1, \dots, i_t, j_1, \dots, j_r]}+j_ri_t\}=\mathcal{G}_{[i_1, \dots, i_t]} $, then the LHS of Eq. (\ref{JQQ-Pi-eq-mod}) is $$\text{LHS}=e(\mathcal{T}_{[i_1, \dots, i_{t-1}]})M_{i_{t-1}i_t}+\sum_{j_1, \dots, j_r}e(\mathcal{T}_{[i_1, \dots, i_{t}, j_1, \dots, j_r]})M_{j_ri_t}=e(\mathcal{G}_{[i_1, \dots, i_t]}).$$ 
 Now we showed the solution $\vPi^{i_1}([i_1, \dots, i_t])=e(\mathcal{T}_{[i_1,\dots, i_t]})$ satisfies the linear equations Eq. (\ref{JQQ-Pi-eq-mod}).

Suppose the derived chain $\eta_t$ can reach $[i_1,\dots, i_t]$ from $[i_1]$, since the original MC is irreducible, $\eta_t$ can also return to $[i_1]$ from $[i_1,\dots, i_t]$. Therefore $[\mathscr{S}]_{i_1}$ is an irreducible class. Moreover, since the period of state $i_1$ is 1 in original MC, the period of state $[i_1]$ in the derived chain is also 1. Because $\Pr(\eta_t=[i_1] | \eta_0=[i_1])=\Pr(X_t=i_1 | X_0=i_1)$. We showed the derived chain is irreducible and has an aperiodic state, thus, the derived chain is ergodic on the class $[\mathscr{S}]_{i_1}$. Then $\eta$ has a unique invariant distribution $\vPi^{i_1}$ on the class $[\mathscr{S}]_{i_1}$, up to the normalization.
 
 Now we have proved $\vPi^{i_1}([i_1, \dots, i_t])\propto e(\mathcal{T}_{[i_1,\dots, i_t]})$.

\bigskip
The second part is to prove the normalization factor is $\Sigma=\sum_{i=1}^n e(\mathcal{T}_i)=e(\cup_i \mathcal{T}_i)$, which is the same as the normalization factor in calculating the invariant distribution of the original MC. By the definition, $\Sigma=\sum_{[i_1, \dots, i_t]\in [S]_{i_1}} e(\mathcal{T}_{[i_1, \dots, i_t]})=e(\cup_{i_2, \dots, i_t}\mathcal{T}_{[i_1, \dots, i_t]})$. 
If a directed rooted tree with root $i_t$, $T\in \mathcal{T}_{i_t}$, there exists a path from $i_1$ to $i_t$. If $t=1$, the path degenerates to a point. So $T\in \cup_{i_2, \dots, i_{t-1}}\mathcal{T}_{[i_1,i_2,\dots,i_{t-1},i_t]}$. If a directed rooted tree $T\in \cup_{i_2, \dots, i_{t-1}}\mathcal{T}_{[i_1,i_2,\dots,i_{t-1},i_t]}$, then $T$ has root $i_t$ and $T\in\mathcal{T}_{i_t}$. 
So we have 
\begin{align} \label{eq-set}
e(\cup_{i_2, \dots, i_{t-1}}\mathcal{T}_{[i_1,i_2,\dots,i_{t-1},i_t]})=e(\mathcal{T}_{i_t}).
\end{align}

That implies the weights of the union of set on $i_t$ are the same, i.e, $e(\cup_{i_t} \mathcal{T}_{i_t})=e(\cup_{i_2, \dots, i_t}\mathcal{T}_{[i_1, \dots, i_t]})$. This proves the normalization factor is $\Sigma=\sum_{i=1}^n e(\mathcal{T}_i)$ and the invariant distribution $\vPi^{i_1}([i_1, \dots, i_t])$ is given by Eq. (\ref{JQQ-Pi}). 

Moreover, if we divide  both hand-side of Eq. (\ref{eq-set}) by $\Sigma$, it gives the following relationship on both invariant distributions by using the pre-image of the homomorphism, $\vPi^{i_1}(f^{-1}(i_t))=\vpi(i_t)$.

\end{proof}

By using Theorem \ref{JQQ-Pi-theo}, we built a bridge between the cycle coordinates $(\mathcal{C}, w_c)$ and the maximum entropy RDS of the MC. 
\begin{corollary}
The mean number of occurrences of the cycle $c$ per step, $w_c$ is 
\begin{align}\label{JQQ-wc}
w_c=\frac{Q(\alpha: \mathcal{A}(\alpha)=c)}{\mathbb{E}^Q(\|\mathcal{A}(\alpha)\|)}, 
\end{align}
where $Q(\alpha: \mathcal{A}(\alpha)=c)$ is the probability of the deterministic map with single attractor $c$ under maximum entropy RDS. 
\end{corollary}
\begin{proof}
From the previous discussion, if $c=(i_1,\dots, i_t)$, $w_c=\frac{e(\mathcal{T}_{[i_1,\dots, i_t]})M_{i_ti_1}}{\Sigma}$. The numerator $e(\mathcal{T}_{[i_1,\dots, i_t]})M_{i_ti_1}$ is equal to the weight of the set of directed graph with the only one limit cycle $(i_1, \dots, i_t)$ and it can also be interpreted as the probability of the deterministic map with single attractor $c$, i.e, $Q(\alpha: \mathcal{A}(\alpha)=c)$. The denominator $\Sigma=\mathbb{E}^Q(\|\mathcal{A}(\alpha)\|)$ is proven in Proposition \ref{sigma-rds}.
\end{proof}

Instead of dividing the step size $t$ in finding the cycle weight $w_c$, divide the total number of cycles formed and we will have the following ergodic theorem. Though the proof is straightforward and elementary, the result in fact is very interesting. 
\begin{corollary} \label{cycle-ergodic}
The frequency of the cycle $c$ along the sample sequence almost surely converges to the probability of the deterministic map with the attractor $c$ given the attractor of the map is single, 
\begin{align}
\lim_{t\rightarrow +\infty} \frac{w_{c,t}(\omega)}{\sum_{c\in \mathcal{C}} w_{c,t}(\omega)}=\vp_c, \text{\ almost surely, }
\end{align}
where $\vp_c=Q(\alpha: \mathcal{A}(\alpha)=c\ |\ \mathcal{A}(\alpha) \text{\ is single})$.
\end{corollary}

\begin{proof}
We know $\lim_{t\rightarrow +\infty}w_{c,t}(\omega)/t=w_c$ almost surely. After using Eq. (\ref{JQQ-wc}) we have 
\begin{align*}
\lim_{t\rightarrow +\infty}\frac{w_{c,t}(\omega)/t}{\sum_{c\in \mathcal{C}} w_{c,t}(\omega)/t}=\frac{w_c}{\sum_{c\in \mathcal{C}}w_c}=\frac{Q(\alpha: \mathcal{A}(\alpha)=c)}{\sum_{c\in \mathcal{C}}Q(\alpha: \mathcal{A}(\alpha)=c)}.
\end{align*}
\end{proof}

So $\sum_{c\in \mathcal{C}}w_c$ is the mean number of occurrences of any possible cycle per step. Then the reciprocal of this quantity, denoted $\lambda=\frac{1}{\sum_{c\in \mathcal{C}}w_c}$, will be the mean steps to generate a cycle and is the ``time unit'' of forming the cycle.

\begin{corollary}
 \begin{align}
 \lim_{t\rightarrow +\infty}\frac{\sum_{c\in \mathcal{C}}w_{c,t}(\omega)\|c\|}{\sum_{c\in \mathcal{C}}w_{c,t}(\omega)}=\lambda \ \text{almost surely.}
 \end{align} 
\end{corollary}
\begin{proof}
We know $\sum_{c\in\mathcal{C}}w_c\|c\|=1$.
\begin{align*}
\lim_{t\rightarrow +\infty}\frac{\sum_{c\in \mathcal{C}}w_{c,t}(\omega)\|c\|/t}{\sum_{c\in \mathcal{C}}w_{c,t}(\omega)/t}=\frac{\sum_{c\in\mathcal{C}}w_c\|c\|}{\sum_{c\in \mathcal{C}}w_c}=\frac{1}{\sum_{c\in \mathcal{C}}w_c}.
\end{align*}
\end{proof}

For each sample sequence $X_t(\omega)$, it induces a sequence of cycles and $w_{c,t}(\omega)$ counts the number of cycles $c$ occurred up to $t$. Instead of studying $X_t(\omega)$, we try to study the dynamics of cycles. Unfortunately, the dynamics of cycle is not Markovian and seemingly complicated, but it is still ergodic with invariant distribution $\vp_c$ from Corollary \ref{cycle-ergodic}. 
The cycle weight $w_c$ can be express by $w_c=\vp_c/\lambda$.

With the expression of the cycle weight $w_c$, now we can connect the cycle coordinates with the edge coordinates $M_{ij}$ and invariant distribution $\vpi$.
 
\begin{corollary}
\begin{align} \label{circ-dist}
\sum_{c\in \mathcal{C}}w_cJ_c(i)=\vpi_i, \ \  \sum_{c\in \mathcal{C}}w_c J_c(i,j)=\vpi_i M_{ij},
\end{align}
where $J_c(i)=\begin{cases}1 & \text{if $i\in c$}\\ 0 & \text{Otherwise}\end{cases}$, $J_c(i,j)=\begin{cases}1 & \text{if $ij$ is an edge of $c$}\\ 0 & \text{Otherwise}\end{cases}$. 
\end{corollary}
\begin{proof}
The RHS of the first equality is 
$$\sum_{c\in \mathcal{C}}w_cJ_c(i)=\lim_{t\rightarrow+\infty}\frac{\sum_{c\in \mathcal{C}}w_{c,t}J_c(i)}{t}.$$
The numerator is the number of occurrences of the state $i$ up to $t$ without counting the derived chain. But the maximum length of the derived chain is $n$. 
$$\frac{\# \text{state\ } i-n}{t} \le\frac{\sum_{c\in \mathcal{C}}w_{c,t}J_c(i)}{t}\le \frac{\# \text{state\ } i}{t}.$$ Taking the limit $t\rightarrow +\infty$, it gives $\sum_{c\in \mathcal{C}}w_cJ_c(i)=\vpi_i$.

Since $ij$ is an edge of $c$, we can write it out RHS of the second equality explicitly, 
\begin{align*}
\sum_{c\in \mathcal{C}}w_c J_c(i,j)=\sum_{i_1, \dots, i_t}\frac{e(\mathcal{T}_{[j, i_1, \dots, i_t, i]})}{\Sigma}M_{ij}=\vpi_i M_{ij}.
\end{align*}
From Eq. (\ref{eq-set}), $\sum_{i_1, \dots, i_t}\frac{e(\mathcal{T}_{[j, i_1, \dots, i_t, i]})}{\Sigma}=\frac{e(\mathcal{T}_i)}{\Sigma}=\vpi_i$.
\end{proof}

\begin{theorem}
The entropy production rate $e_p$ is represented by the cycle coordinates $(\mathcal{C}, w_c)$ and furthermore, by the invariant distribution of cycle dynamics $\vp_c$ and $\lambda$ 
\begin{align}
&e_p=\sum_{c\in\mathcal{C}}w_c\log\frac{w_c}{w_{c_-}},\ \ \ e_p=\frac{H(\vp_c, \vp_{c_-})}{\lambda}.
\end{align}
\end{theorem}
\begin{proof}
Use the second equality of Eq. (\ref{circ-dist}) in the expression of entropy production rate in (\ref{epr}), assume the cycle $c=(i_1, \dots, i_t)$. Due to the Eq. (\ref{eT-express}), we have $w_c=M_{i_1i_2}\dots M_{i_ti_1}\frac{e(\mathcal{T}_{i_1,\dots, i_t})}{\Sigma}$ and $w_{c_-}=M_{i_1i_t}\dots M_{i_2i_1}\frac{e(\mathcal{T}_{i_1,\dots, i_t})}{\Sigma}$. 

\begin{align*} 
e_p=\sum_{i,j} \Big(\sum_{c\in \mathcal{C}}w_c J_c(i,j)\Big)\log \frac{M_{ij}}{ M_{ji}}= \sum_{c\in \mathcal{C}}w_c \log\frac{M_{i_1i_2}\dots M_{i_ti_1}}{M_{i_1i_t}\dots M_{i_2i_1}}=\sum_{c\in \mathcal{C}}w_c \log\frac{w_c}{w_{c_-}}.
\end{align*}

Use the expression we get in Eq. (\ref{JQQ-wc})
\begin{align*}
e_p&=\sum_{c\in \mathcal{C}} w_c\log \frac{w_c}{w_{c_-}}=\sum_{c\in \mathcal{C}}\frac{Q(\alpha: \mathcal{A}(\alpha)=c)}{\mathbb{E}^Q(\|\mathcal{A}(\alpha)\|)}
\log\frac{Q(\alpha: \mathcal{A}(\alpha)=c)}{Q(\alpha: \mathcal{A}(\alpha)=c_-)} \\
&=\frac{\sum_{c'\in\mathcal{C}}Q(\alpha: \mathcal{A}(\alpha)=c')}{\mathbb{E}^Q(\|\mathcal{A}(\alpha)\|)} \sum_{c\in\mathcal{C}} \frac{Q(\alpha: \mathcal{A}(\alpha)=c)}{\sum_{c'\in\mathcal{C}}Q(\alpha: \mathcal{A}(\alpha)=c')}\log\frac{Q(\alpha: \mathcal{A}(\alpha)=c)}{Q(\alpha: \mathcal{A}(\alpha)=c_-)}\\
&= \Big(\sum_{c'\in\mathcal{C}} \frac{Q(\alpha: \mathcal{A}(\alpha)=c')}{\mathbb{E}^Q(\|\mathcal{A}(\alpha)\|)}\Big)\sum_{c\in\mathcal{C}}\vp_c\log\frac{\vp_c}{\vp_{c_-}}=\Big(\sum_{c'\in \mathcal{C}}w_{c'}\Big) H(\vp_c, \vp_{c_-})=\frac{H(\vp_c, \vp_{c_-})}{\lambda}.
\end{align*}
\end{proof}
From the theorem, entropy production rate $e_p$ is proportional to the relative entropy of the invariant distribution of cycle dynamics with respect to its reverse cycle. The extra constant term $1/\lambda$ is the average number of cycle occurrences per step which bridges from the cycle dynamics back to the original MC. 

\subsection{Entropy Production of Doubly Stochastic MC
and its invertible RDS}  Entropy production characterizes
dynamic randomness, which can be divided conceptually
as ``uncertainties in the past''  and ``uncertainties in the
future''.  The former is represented by non-invertible,
``many-to-one'' maps while the latter is best represented 
by stochastic, ``one-to-many'' dynamics.  In connection
to the entropy production in non-invertible dynamics, 
Ruelle has introduced the notion of {\em folding entropy}
\cite{ruelle-jsp}.

	In Sec. \ref{sec-4.3}, the entropy production is discussed in terms of maximum entropy representation of the MC. Here we establish
a relationship between the discrete state, finite RDS with only
invertible transformations, {\em invertible RDS}, and the
entropy production rate of its corresponding MC, which is 
always doubly stochastic:  Both the rows and columns 
of the MC transition matrix sum to 1. Thus the invariant distribution is uniform, i.e, $\vpi_i=1/n$. From the Birkhoff-Von Neumann theorem, the set of doubly stochastic matrices is the convex hull of the set of $n\times n$ permutation matrices, and the vertices are precisely permutation matrices  \cite{ywq}. In other words, for every
doubly stochastic MC, one can assign the probability measure $Q(\alpha)$ on the the set of invertible maps $\Delta$, 
such that $\sum_{\alpha\in \Delta}Q(\alpha)=1$ and $Q(\alpha: \alpha(i)=j)=M_{ij}$.  It provides an  invertible RDS representation for a doubly stochastic MC through an i.i.d. process.   Such representation is different from the maximum entropy representation, and it may not be unique.  For the 
invertible RDS, each invertible map $\alpha$ has a well-defined 
inverse map $\alpha^{-1}$;
its matrix representation is the inverse of its permutation matrix,
$P_{\alpha^{-1}}=P_{\alpha}^{-1}$. So the cycle of the inverse map is reversed compared with the original map. 

	With this set up, one can introduce a {\em time-reversal dual} probability measure on the set of permutation matrices $\Delta$, $Q^-(\alpha)=Q(\alpha^{-1})$. The probabilities of the invertible map and its inverse are flipped. $Q^{-}$ define a dual RDS through an i.i.d. process, and its corresponding MC is exactly the time-reversed process with transition matrix $M^-_{ij}=M_{ji}$.

\begin{theorem}
\label{th-4.6}
The entropy production rate of a doubly stochastic MC has an upper bound in terms of the relative entropy of measure $Q$ with respect to $Q^-$. 
\begin{align}
e_p\le H(Q, Q^-).
\end{align}
The equality holds if and only if $Q=Q^-$.
\end{theorem}

\begin{proof}
The entropy production rate $e_p=\frac{1}{n}\sum_{ij} M_{ij}\log\frac{M_{ij}}{M_{ji}}$. Based on the log-sum inequality, which states that for two sets of non-negative numbers $(a_1, \dots, a_n)$ and $(b_1, \dots, b_n)$,  
$$ \sum_{i=1}^n a_i\log\frac{a_i}{b_i}\ge \Big(\sum_{i=1}^n a_i\Big)\log\frac{(\sum_{i=1}^n a_i)}{
(\sum_{i=1}^n b_i)}.$$
The equality holds if and only if $a_i=b_i$. 
\begin{align}\nonumber
e_p&=\frac{1}{n}\sum_{ij}  \Big(\sum_{\alpha: \alpha(i)=j}Q(\alpha)\Big)\log\frac{  \sum_{\alpha: \alpha(i)=j}Q(\alpha)}{  \sum_{\alpha: \alpha(j)=i}Q(\alpha)} \\ \nonumber
&=\frac{1}{n}\sum_{ij}  \Big(\sum_{\alpha: \alpha(i)=j}Q(\alpha)\Big)\log\frac{  \sum_{\alpha: \alpha(i)=j}Q(\alpha)}{  \sum_{\alpha: \alpha(i)=j}Q^-(\alpha)} \\ \nonumber
&\le \frac{1}{n}\sum_{ij} \sum_{\alpha: \alpha(i)=j}\Big(Q(\alpha)\log\frac{ Q(\alpha)}{  Q^-(\alpha)}\Big)\\  \label{last-eq}
&=\sum_{\alpha}Q(\alpha)\log\frac{ Q(\alpha)}{  Q^-(\alpha)}.
\end{align}
The last equality (\ref{last-eq}) is because the term $Q(\alpha)\log\frac{ Q(\alpha)}{  Q^-(\alpha)}$ sums up exactly $n$ times for each $\alpha$ which is the number of edges. The equality holds if and only if $Q(\alpha)=Q^-(\alpha)$. 
\end{proof}
The condition that the equality holds is the sufficient condition of the doubly stochastic MC being detailed balance, i.e, $M^-=M$.  

Theorem \ref{th-4.6} should be compared and contrasted
with an earlier result from Kifer and Ye {\em et al.} \cite{kifer2012random,ywq}: $h_{\text{RDS}}\ge h_{\text{MC}}$, i.e, the metric entropy of MC in (\ref{mc-m-entropy}) is upper bounded by the metric entropy of RDS, $h_{\text{RDS}}=-\sum_{\alpha}Q(\alpha)\ln Q(\alpha)$. 
$h_{\text{RDS}}$ exists unique finite upper bound which is maximum entropy RDS, but $H(Q,Q^-)$ here could be $+\infty$.

\bibliographystyle{plain}

\bibliography{RDS}

\appendix

\section{Proof of Theorem \ref{Hill-theorem}} \label{appendix-A}
\begin{proof}
Consider the equation to solve $\sum_{k}\vpi_k M_{kj}=\vpi_j$. It is equivalent with solving the following equation
\begin{align}
\sum_{\substack{k=1\\ k\ne j}}^n \vpi_k M_{kj}=\vpi_j-\vpi_j M_{jj}=\sum_{\substack{i=1\\ i\ne j}}^n M_{ji} \vpi_j
\end{align} 
Let $\mathcal{G}_j$ be the set of directed graphs that have exactly one limit cycle and $j$ is contained in that cycle. If a directed rooted tree $T\in \mathcal{T}_j$, adding the edge $j \rightarrow i$ will create an element $G\in \mathcal{G}_j$ and $e(G)=M_{ji}e(T)$. If a directed graph $G\in \mathcal{G}_j$, deleting the edge $j\rightarrow i$ will create an element $T\in \mathcal{T}_j$ and $M_{ji}e(T)=e(G)$. So we have $\sum_{\substack{i=1\\ i\ne j}}^n M_{ji}e(\mathcal{T}_j)=e(\mathcal{G}_j)$. Here we are considering the outgoing edge from $j$ and we can consider the incoming edge $k\rightarrow j$ as well. Similarly, $\sum_{\substack{k=1\\ k\ne j}}^n M_{kj}e(\mathcal{T}_k)=e(\mathcal{G}_k)$. So $\vpi_j \propto e(\mathcal{T}_j)$. After renormalization, the solution (\ref{Hill-pi}) is the invariant distribution. 
\end{proof}
We will give one example for the theorem. 

{\bf Example:} If the MC has the transition matrix 
\[
    M=\begin{bmatrix} M_{11} & M_{12} & M_{13} \\ M_{21} & M_{22} & 0\\ M_{31} &0 &M_{33} \end{bmatrix},
\]
each set $\mathcal{T}_i$ has only one element. The directed rooted trees and their weights are shown in Fig. \ref{Ex-pi}. So the invariant distribution is 
\begin{align}
\vpi=\left(\frac{M_{21}M_{31}}{\Sigma}, \frac{M_{12}M_{31}}{\Sigma}, \frac{M_{21}M_{13}}{\Sigma}\right)
\end{align}
where $\Sigma=M_{21}M_{31}+M_{12}M_{31}+M_{21}M_{13}$.

\begin{figure}
\includegraphics[scale=0.55]{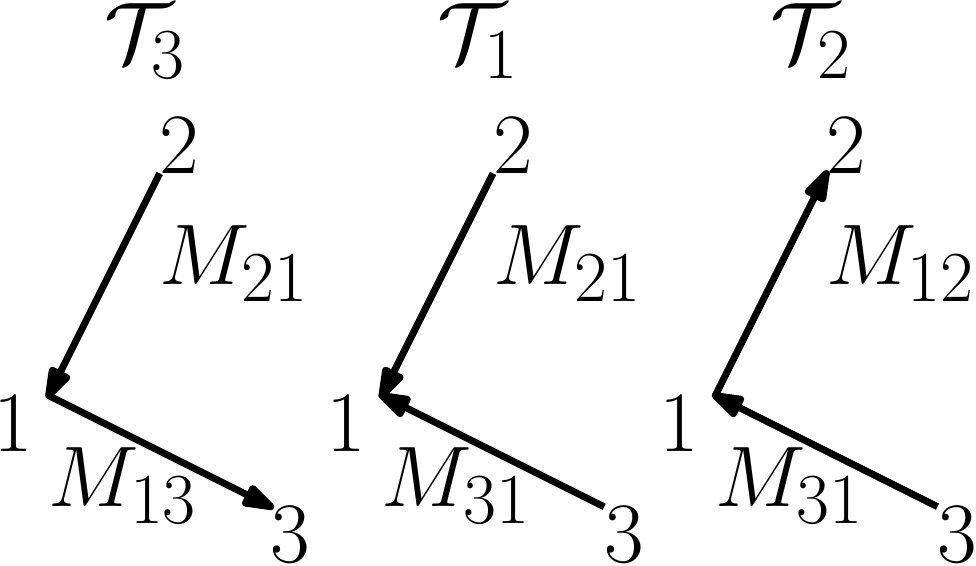}
\caption{$e(\mathcal{T}_1)=M_{21}M_{31}, e(\mathcal{T}_2)=M_{12}M_{31}, e(\mathcal{T}_3)=M_{21}M_{13}$}
\label{Ex-pi}
\end{figure}

\section{Matrix-Tree Theorem} \label{appendix-B}

The matrix-tree theorem is a refined formula that gives the complete symbolic series for directed rooted trees with specified roots and more generally for forests with specified roots \cite{Bollobas}. We introduce variable $M_{ij}$ for all $i,j\in\mathscr{S}$ and define the monomial $x_T$ for the directed rooted trees $T$ to be the product of the variables $M_{ij}$ for all directed edges $i\rightarrow j$ in $T$. For the example above in Fig. \ref{Ex-pi}, they are 
\begin{align}
x_{T_1}=M_{21}M_{31}, x_{T_2}=M_{12}M_{31}, x_{T_3}=M_{21}M_{13}
\end{align}
Note the weight of the trees we defined before is exactly evaluated in the monomial for given $M_{ij}$. Moreover, the directed rooted tree are determined by its monomial $x_T$. Similarly, it can be extended to rooted forests. 

Given a subset $I\subset\mathscr{S}$, we define $F_{n, I}$ to be the sum of the monomials for all forests $G$ whose set of roots is  $I$, which is called the generating function for $G$. 
\begin{align}
F_{n, I}=\sum_{G: \text{roots}(G)=I}x_G. 
\end{align}

The matrix-tree theorem is stated as follows, 
\begin{theorem}
In MC, the generating function $F_{n, I}(M)$ for all forests rooted at $I$, with edges directed towards the roots, is given by the determinant 
\begin{align}
F_{n, I}(M)=\det D(\{I\}^c)
\end{align}
where $D=I-M$ and $D(\{I\}^c)$ is the submatrix of the matrix $D$ by deleting the rows and columns with indices $i\in I$.
\end{theorem}
In the previous example, the matrix $D$ is 
\[
   \begin{bmatrix} M_{12}+M_{13} & -M_{12} & -M_{13} \\ -M_{21} & M_{21} & 0\\ -M_{31} &0 &M_{31} 
\end{bmatrix}.
\]
If $I=\{2\}$, then $\det D(\{2\})$ is $(M_{12}+M_{13})M_{31}-M_{31}M_{13}=M_{12}M_{31}$. It agrees with the expression for $x_{T_2}$. 

The generating function $F_{n,\{i\}}(M)$ evaluated at given transition matrix $M$ is equal to $e(\mathcal{T}_i)$ the weight of the set of directed rooted trees whose root is state $i$. Then the Eq.
(\ref{Hill-pi}) can be rewritten as 
\begin{align}
\vpi_i=\frac{D(\{i\}^c)}{\sum_i D(\{i\}^c)}.
\end{align}
\end{document}